\documentclass[12pt]{article}

\usepackage[latin1]{inputenc}
\usepackage{amssymb}
\usepackage{amsmath}
\usepackage{amsthm}
\usepackage{mathtools}
\usepackage{tikz}
\usepackage{import}

\newcommand{\inv}{^{-1}}

\DeclareMathOperator{\dist}{d}

\newtheorem{thm}{Theorem}[section]
\newtheorem{cor}[thm]{Corollary}
\newtheorem{lem}[thm]{Lemma}
\newtheorem{prop}[thm]{Proposition}
\newtheorem{prob}[thm]{Problem}
\newtheorem{que}[thm]{Question}

\theoremstyle{definition}
\newtheorem{defn}[thm]{Definition}

\newtheorem{nt}[thm]{Notation}

\theoremstyle{remark}
\newtheorem{rem}[thm]{Remark}
\newtheorem{ex}[thm]{Example}

\newcommand{\cA }{\mathcal A}
\newcommand{\cC }{\mathcal C}

\newcommand{\cL }{\mathcal L}
\newcommand{\cM }{\mathcal M}

\newcommand{\cR }{\mathcal R}
\newcommand{\cS }{\mathcal S}

\newcommand{\bQ}{\mathbb{Q}}
\newcommand{\bM}{\mathbb{M}}

\newcommand{\bR}{\mathbb{R}}

\newcommand{\bT}{\mathbb{T}}
\newcommand{\bZ}{\mathbb{Z}}

\newcommand{\Reg}{\mathsf{Reg}}
\newcommand{\CF}{\mathsf{CF}}
\newcommand{\onecounter}{\mathsf{1C}}
\newcommand{\RecEnum}{\mathsf{RE}}

\newcommand{\WP}{\mathrm{WP}}

\def\coloneqq{\mathrel{\mathop\mathchar"303A}\mkern-1.2mu=}

\date{}
\author{}

\begin{document}

\title{Regular left-orders on groups}
\author{Yago Antol\'{i}n, Crist\'{o}bal Rivas, Hang Lu Su }
\date{\today}

\maketitle

\begin{abstract}
A regular left-order on a finitely generated group  $G$ is a total, left-multiplication invariant order on $G$  whose corresponding positive cone is the image of a regular language over the generating set of the group under the evaluation map. We show that admitting regular left-orders is stable under extensions and wreath products and we give a classification of the groups whose left-orders are all regular left-orders. In addition, we prove that a solvable Baumslag-Solitar group $B(1,n)$ admits a regular left-order if and only if $n\geq -1$. Finally, Hermiller and \u{S}uni\'{c} showed that no free product admits a regular left-order. We show that if $A$ and $B$ are groups with regular left-orders, then $(A*B)\times \mathbb{Z}$ admits a regular left-order.
\end{abstract}

\vspace{0.2cm}\vspace{0.2cm}


\textbf{}
\tableofcontents

\section{Introduction}
\label{sec:intro}

A group $G$ is \emph{left-orderable} if there exists a total order $\prec$ on the elements of $G$ which is invariant under left-multiplication, that is
$$ g \prec h \iff fg \prec fh, \qquad \forall g,h,f \in G.$$
In this case, the relation $\prec$ is called a {\em left-order}. It is sometimes easier to understand left-orders in terms of sets of elements which are greater than the identity.
Such sets are called {\em positive cones} (of the order). A positive cone completely encodes its associated left-order and vice-versa.
Equivalently, a positive cone is a subset $P \subset G$ which is closed under multiplication $PP \subseteq P$,  and partitions $G$ as $G = P \sqcup P^{-1} \sqcup \{1_G\}$ where the union is disjoint. 

This paper discusses the computational complexity  of left-orders (or equivalently positive cones). 
This is a topic that has gained interest in recent years: Darbinyan  \cite{Dar} and Harrison-Trainor \cite{HT} have constructed groups with solvable word problems but no computable left-orders. 
Rourke and Weist have studied the complexity of left-orders on certain mapping class groups \cite{Rourke}. 
\u{S}uni\'{c} \cite{Sunik13,Sunik13b} proved the existence of one-counter left-orders and later, with Hermiller \cite{HS}, proved that \u{S}uni\'{c}'s left-orders on free groups are the computationally the simplest possible by showing that no positive cone on a free product admits a regular left-order.

We are interested in finding groups that admit a computationally simple left-order and in describing such orders.
Specifically, we investigate which finitely generated groups admit a {\em regular left-order} or, equivalently, a {\em regular positive cone}. A positive cone $P\subseteq G$ is called regular if there is a finite generating set $X$ of $G$ and a  regular language $\mathcal L\subseteq X^*$  such that $\cL$ evaluates onto $P$. 
By a {\it  formal language} we mean a subset of a finitely generated free monoid $X^*$. 
 A formal language $\cL \subseteq X^*$ is {\it regular} if it is accepted by a {\it finite state automaton}. A finite state automaton for $\cL$ is a directed $X$-edge-labelled finite graph such that $\cL$ consists of the set of labels of paths from a fixed initial vertex to a vertex in a fixed subset of vertices called accepted states. 
Positive cones which are represented by a regular language are the simplest ones computationally from the point of view of the Chomsky hierarchy of formal languages \cite{Chomsky}. Section \ref{sec:prelim} will be devoted to reviewing the necessary material about formal languages and left-orders.

In previous works, the authors have given examples of groups not admitting regular left-orders \cite{AABR, Su}. 
In this paper we focus instead on constructing regular left-orders.
We  show that regularity of positive cones is preserved under the constructions of the following theorem.

\begin{thm}[Propositions \ref{prop: inheritance subgroups}, \ref{prop:extensions} and \ref{prop:wreath product}]\label{thm: closure}

 The class of finitely generated groups admitting regular positive cones is closed under
 passing to finite index subgroups, extensions and wreath products.
\end{thm}
The inheritance under passing to finite index subgroups is due to H. L. Su \cite{Su}. 
The other two closure properties are proved in Section \ref{sec: extensions} where we study left-orders associated to group extensions.
A basic fact is that the group of additive integers $\bZ$ admits regular left-orders. Starting from this, and the known fact that amenable left-orderable groups are locally indicable \cite{Morris}, the closure under group extensions property provides the following family of examples of groups which admit regular positive cones.
\begin{cor}[Proposition \ref{prop: virtually polycyclic}]\label{cor: polycyclic}
Left-orderable virtually polycyclic groups admit regular positive cones.
\end{cor}

For the case of wreath products $G=N\wr Q$, we have an exact sequence $1\to \oplus_{q\in Q} N \to G \to Q \to \{1\}$ but here $\oplus_{q\in Q} N$ does not admit a regular positive cone as, in general, it is not finitely generated.
Therefore, the construction of regular left-orders on  wreath products differs from the construction for the extensions. 

It is already known that one cannot expect to generalize Corollary \ref{cor: polycyclic} for solvable non-polycyclic groups. 
Indeed, Darbinyan \cite[Theorem 2]{Dar} showed that there is a two-generated recursively presented left-orderable solvable group
of derived length 3 with undecidable word problem (in fact, their order turns out to be two-sided invariant). 
As we will see in the Appendix, this implies that no left-order on this group is computable (i.e. described by a language that can be recognized by a Turing machine).

In the present paper, we are able to decide when a solvable Baumslag-Solitar group $BS(1,q)$ admits a regular positive cone or not.
 Recall that $BS(1,q)=\langle a,b \mid aba^{-1} =b^q\rangle$ and it is isomorphic to $\bZ[1/q]\rtimes \bZ$.
Note that $q\in \bZ$ and for $q=0$  we adopt the convention that $BS(1,0)=\bZ$.

\begin{thm}[Theorem \ref{thm: BS}] For all $q\in \mathbb{Z}$, the solvable Baumslag-Solitar group $BS(1,q)$ admits a one-counter left-order. Moreover,  $BS(1,q)$ admits a regular left-order  if and only if $q\geq -1$. 
\end{thm}
The class of one-counter languages  generalizes the class of regular languages.
These are languages that can be recognized by finite state automaton (the same machine as for regular languages) equipped with a very simple memory which consists on a stack that contains copies of the same letter (the counter).
The machine, depending on the input and whether the stack is empty or not, can modify the stack by adding or removing a counter from the stack. 
See Section \ref{sec:prelim} for a formal definition. 

From this theorem, we observe that while the existence of  regular of left-orders passes to finite index subgroups (Theorem \ref{thm: closure}),  it is not true that it passes to finite-index overgroups (even if the overgroup is left-orderable).
Indeed, $BS(1,-2)$ is a left-orderable group that does not admit any regular left-order but  contains  $BS(1,4)$ as an index 2 subgroup that does admit regular left-orders. 
Therefore,  admitting a regular order is not a property preserved by commensurability among left-orderable groups. 

The reason why $BS(1,q)$ for $q<-1$ does not admit regular positive cones is a combination of an algebraic argument and a geometric argument.
The algebraic argument goes back to Tararin \cite{Tararin} (see also \cite{rivas jgt}) and describes all possible left-orders in these groups.
The geometric argument uses Bieri-Neumann-Strebel theory \cite{BNS} and implies that positive cones associated to the orders described by Tararin are not coarsely connected (see Definition \ref{defn: coarsely connnected}) which is a necessary condition for regularity (see Lemma \ref{lem: regular implies coarsely connected}). 

We will use Tararin's classification \cite{Tararin} of groups admitting finitely many left-orders to describe the groups  admitting only regular left-orders. More precisely, in Section \ref{sec: all left-orders are regular} we prove the following theorem.

\begin{thm}[Theorem \ref{thm:tararin}]\label{thm: all regular}
 A finitely generated group $G$ admits only regular left-orders  if and only if it is Tararin poly-$\bZ$, that is, it   admits a unique subnormal series 
 $$G = G_0 \unrhd G_1 \unrhd \dots  \unrhd G_n = \{1\}$$
 where $G_i/G_{i+1}\cong \bZ$ and $G_i/G_{i+2}\cong K=\langle a,b \mid aba^{-1} = b^{-1}\rangle$.
\end{thm}

Finally, the last construction that preserves left-orderability that we will discuss are free products \cite{Passman}. 
We know from the work of Hermiller and \u{S}uni\'{c} that no non-trivial free product (in particular free groups) admits a regular positive cone. 
One might think that then it is helpless to pursue this route, however we have the following.
\begin{thm}[Corollary \ref{cor: A*BxZ}]\label{thm: A*BxZ} Suppose that $A$ and $B$ are groups admitting regular left-orders. 
Then $(A*B)\times\bZ$ admits a regular left-order.
\end{thm}
This sort of phenomenon was already observed in \cite{Su}, where H.L. Su found a finitely generated positive cone for $F_2\times \bZ$ (and hence a regular positive cone). 
In Section \ref{sec: quasimorphism} we prove Theorem \ref{thm: A*BxZ} and give an explanation of the role of the factor $\bZ$.
To do so, we will define the concept of an {\it ordering quasi-morphism}.
A quasi-morphism on $G$ is a map $\phi\colon G\to \mathbb{R}$ that is close to a homomorphism in the sense that there is a constant $R$ such that $|\phi(g)+\phi(h)-\phi(gh)|\leq R$. 
An ordering quasi-morphism  $\phi$ (Definition \ref{defn: ordering quasi-morphism}) has the property that elements whose image under $\phi$ is positive form a (relative) positive cone (see Definition \ref{defn: rel-order}. This will allow to construct a left-order on the group, as we will see.

While we coin the name of ordering quasi-morphism in this paper, the original idea comes from the work of \u{S}uni\'{c} \cite{Sunik13,Sunik13b}  about orders on free groups,  Dicks and \u{S}uni\'{c} \cite{DicksSunic2014}  about orders on free products  and Antol\'{i}n, Dicks and \u{S}uni\'{c}  \cite{ADS} about orders on certain fundamental groups of graph of groups.
We show that when the ordering-quasimorphism can be computed with a transducer (another type of machine similar to a finite state automaton, see Definition \ref{defn: rational transducer}) the corresponding left-order is one-counter.
More precisely, our chosen function $\phi$ has image $\bZ$, and when embedding $G$ in $G\times \bZ$, we use the $\bZ$ factor to compensate the increments/decrements in $\phi$ so that we no longer need a stack to keep track of the value of $\phi$ on the elements of $G$, hence constructing the desired regular left-order on $G\times \bZ$.

One of the motivations for studying left-orders on free products is to get a better understanding of the influence of negative curvature on the complexity of positive cones. 
In particular, we are interested in the following.
\begin{que}
Is there a non-elementary hyperbolic group admitting a regular positive cone?
\end{que} 
This question has some history.  Hermiller and \u{S}uni\'{c} \cite{HS} proved that  non-abelian free groups do not admit regular positive cones. 
Su \cite{Su} proved that acylindrically hyperbolic groups do not have positive cones that are quasi-geodesic and regular. This formalized and strengthened an argument sketched by Calegari \cite{Calegari}. Finally, Alonso, Antol\'{i}n, Brum and Rivas \cite{AABR} proved that non-abelian limit groups (in particular free groups and surface groups)  do not admit coarsely connected positive cones, and therefore do not admit regular positive cones.

Nevertheless, 
if $G$ is  hyperbolic (finitely generated free)-by-$\mathbb{Z}$, then the Bieri-Neumann-Strebel theory guarantees that $G$ admits coarsely connected left-orders (see discussion before Lemma \ref{lem: BNS}).
Thus, one might initially think that those left-orders on hyperbolic (f.g. free)-by-$\bZ$ groups might be regular. As a consequence of the methods developed to prove Theorem 
\ref{thm: A*BxZ} we get that this is almost the case.

\begin{thm}[Theorem \ref{thm: free-by-ZxZ}]\label{thm: lex on free-by-cyclic}
Let $G$ be  (finitely generated free)-by-$\bZ$ group. Then $G\times \bZ$ admits a regular positive cone, whose restriction to $G$ is  coarsely connected, one-counter, but  not regular.
\end{thm}
We conclude Section \ref{sec: quasimorphism} showing that certain Artin groups admit regular positive cones.

Finally, we close the paper with an appendix. Let $X$ be finitely generating set of a group $G$ and $\pi\colon X^*\to G$ be a monoid surjection. If $P$ is a positive cone, one can ask about the complexity of $\pi^{-1}(P)$ as a formal language, which is in more direct analogy with studying formal languages with the Word Problem. We explore this question for regular and context-free languages (the lowest ones in the Chomsky hierarchy), and show that there are extremely few examples of such left-orders (for instance, it is never the case that the language of preimages of a positive cone is regular). This gives partial justification as to why we have focus on our definition of regular positive cones.

\section{Preliminaries about languages and left-orderability}
\label{sec:prelim}
In this section we review some formal language theory needed in the paper, focusing on the classes of regular and context-free languages. 
We also review the definition of left-orderability and we explain what we mean by a regular or context-free left-order on a group.
Most of the content of this section is standard or well-known and the main purpose is to set the notation and the definitions for the rest of the paper.

\subsection{Review of abstract families of languages}
By a {\it formal language} we refer to a subset of $\cL$ of $X^*$, where $X$ is a finite set and $X^*$  denotes the set of  all finite sequences over $X$, also called {\it words}.
We denote the empty word by $\epsilon$.
Note that $X^*$ is a monoid with respect to concatenation. We denote by $X^+$ the set $X^*\setminus\{\epsilon\}$, which is a semi-group with concatenation.

A {\it non-deterministic finite state automaton} is a 5-tuple $\bM=(\cS,X,\delta, s_0,\cA)$, where $\cS$ is a finite set whose elements are  called {\it states}, $\cA$ is a subset of $\cS$ of whose states are called {\it accepting states},  $s_0$ is a distinguished element of $\cS$ called {\it initial state}, $X$ is a finite set called the {\it input alphabet}, and $\delta\colon \cS\times (X\sqcup \{\epsilon\})\to \textrm{Subsets}(\cS)$ is a function called the {\it transition function}.

We extend $\delta$ to a function $\delta\colon \cS\times (X\sqcup \{\epsilon\})^*\to \textrm{Subsets}(\cS)$ recursively, by setting $\delta(s,wx)=\cup_{\sigma \in \delta(s,w)} \delta( \sigma ,x)$ where $w\in X^*$, $x\in X$ and $s\in \cS$. 
A language $\cL$ is {\it regular} if there is a non-deterministic finite state automaton $\bM$ 
such that $$\cL= |\bM|:= \{w\in X^* \mid \delta(s_0,w)\cap \cA\neq \emptyset\}.$$

We denote by $\Reg$ the class of all regular languages.

The outputs of $\delta$ under input $\epsilon$  should be thought as the machine acting spontaneously.
Images with input $\epsilon$ are usually called {\it $\epsilon$-moves} or {\it $\epsilon$-transitions}. 
Informally, the machine $\bM$ at a certain step is at a state $s \in \cS$ and can change to another state  $s'\in\cS$  either because it reads some $x\in X$, and $s' \in \delta(s,x)$ or is at state $s'$ spontaneously, meaning that $s' \in  \delta(s,\epsilon)$. 

\begin{rem}\label{rem: different types of fsa}
The class $\Reg$ coincides with the class of languages accepted by non-deterministic finite state automata without $\epsilon$-moves, and even with the class of languages accepted by {\it deterministic finite state automata} (See \cite[Chapter 2]{HopUll}). 
Deterministic automata do not have $\epsilon$-moves and in that setting $\delta(s,x)$ is always a singleton. 
Depending on our needs, we will choose to allow $\epsilon$-moves or not (therefore choosing whether we impose determinism).
\end{rem}

\begin{rem}\label{rem: fsa}
A finite state automaton $\bM$ can be regarded as a labelled directed graph. 
The vertices are the states of the  automaton, and edges are given by the transition function $\delta$. 
There is an edge from $s\in \cS$ to $s'\in \cL$ labeled by $x\in X\sqcup \{\epsilon\}$ if and only if $s'\in\delta(s,x)$.
A word $w$ is accepted by $\bM$ if it is a label of a path from $s_0$ to some accepting state $\cA$.
\end{rem}

\begin{nt}\label{not: counting letters}
Given a word $w\in X^*$, and $x\in X$, we use $\sharp_x (w)$ to denote the number of times the letter $x$ appears in the word $w$.
\end{nt}

\begin{ex}\label{ex: regular}
Let $\bM$ be the finite state automaton given in  Figure \ref{fig: fsa}.
The automaton is described as a graph following Remark \ref{rem: fsa}.
\begin{figure}[ht]
\begin{center}
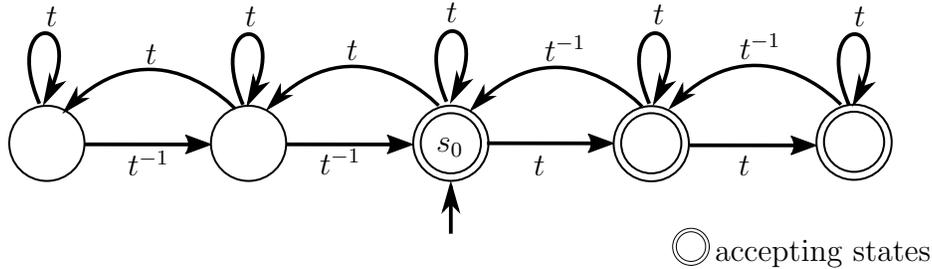
\end{center}
\caption{Finite state automaton accepting the language of Example \ref{ex: regular}.}
\label{fig: fsa}
\end{figure}

The language accepted by $\bM$ in Figure \ref{fig: fsa}, consists on all the $w$ in  $\{w\in \{t^{-1},t\}^* \mid \sharp_{t}(w)\geq \sharp_{t^{-1}}(w)\}$ such that for all subword $u$ of $w$,  $\sharp_{t}(u)-\sharp_{t^{-1}}(u)\geq -2$ holds.
\end{ex}

A {\it (non-deterministic) pushdown automaton} is a 7-tuple $\bM=(\cS,X, \Sigma, \delta,s_0, u_0,\cA)$, where $\cS$ is a finite set whose elements are called {\it states}, $\cA$ is a subset of $\cS$ whose states are called {\it accepting states}, $s_0\in \cS$ is a distinguished element called {\it initial state},  $X$ is a finite set called the {\it input alphabet}, $\Sigma$ is a finite set called the {\it stack alphabet}, $u_0$ is a distinguished word over $\Sigma$ called the {\it initial stack}, and $\delta$ is  a {\it non-deterministic transition function} $$\delta\colon  \cS \times (X\sqcup \{\epsilon\}) \times (\Sigma \sqcup \{\epsilon\}) \to  \textrm{Finite subsets of }(\cS \times \Sigma^*).$$

The pushdown automaton $\bM$ at each stage of the computation is at a certain state and contains a certain word over $\Sigma$ which we call {\it stack word}. 
At the start stage, the state is $s_0$ and the stack word is $u_0$.
To describe how the stack words change, it is convenient to extend $\delta$ so  we allow stack words (and not only stack letters) as inputs.
The behavior of $\bM$ will depend only on the last letter $\sigma$ of the stack word (if the stack word is empty we take $\sigma= \epsilon$).
Precisely, we extend $\delta$ to a function $\delta\colon  \cS \times (X\sqcup \{\epsilon\}) \times \Sigma^* \to \textrm{Finite subsets of }(\cS \times \Sigma^*)$, where if we write $w=u\sigma$ with $\sigma\in \Sigma$, then $\delta(s,x,w)=\{(s',u u') \mid (s',u')\in \delta(s,x,\sigma)\}$. Note that $\delta$ on input word $w=u\sigma \in \Sigma^*$ operates by deleting $\sigma$  and appending $u'$ to the end to get the  word $uu'$ as output.

As in the case of finite state automata, we extend $\delta$ to a function $\delta\colon \cS\times (X\sqcup\{\epsilon\})^*\times \Sigma^*\to \textrm{Finite subsets of }(\cS\times \Sigma^*)$ recursively.

A language $\cL$ is {\it context-free} if there is a non-deterministic pushdown automaton such that 
$$\cL=|\bM|:=\{w\in X^* \mid  \exists (s,u) \in \delta(s_0,w, u_0) \text{ with } s\in \cA\}.$$

We denote by $\CF$ the class of all context-free languages.

A language $\cL$ is {\it one-counter} if there is a non-deterministic pushdown automaton with $|\Sigma|=1$ and such that $\cL=|\bM|$.
The class of one-counter languages will be denoted by $\onecounter$.

\begin{rem}
In many computer science books, they use a model of pushdown automaton that has access to a tape where they write the stack word and read the last letter from it. 
In that model, the stack alphabet must include an extra symbol to denote the blank space on the tape. 
\end{rem}

We clearly have 
$$\Reg \subseteq \onecounter \subseteq \CF.$$
Note that $\Reg$ and $\CF$ are the simplest complexities in the Chomsky hierarchy \cite{Chomsky}.

\begin{ex}\label{ex:one-counter}
Consider the language $\cL=\{w\in \{t^{-1},t\}^* \mid \sharp_t(w)>\sharp_{t^{-1}}(w)\}$.
This language is in the class $\onecounter$ and will appear many times in the paper. 
For completeness, we give in Figure \ref{fig:one-counter} a one-counter push-down automaton for $\cL$ with the graphical notation found in Hopcroft, Motwani and Ullman \cite[Section 6.1.3]{HopUll}.
The vertices of the graph represent the states of the automaton.
An arrow from state $s$ to state $s'$ has label $x,\alpha/\beta$, where $x\in X=\{t,t^{-1}\}$, $\alpha\in \Sigma\cup \{\epsilon\}$ and $\beta\in  \Sigma^*$, if and only if $(s', \beta)\in \delta(s,x,\alpha)$. That is, the label tells what input is used and gives the old and new suffixes of the stack word. 

\begin{figure}[ht]
\hspace{4.1cm}
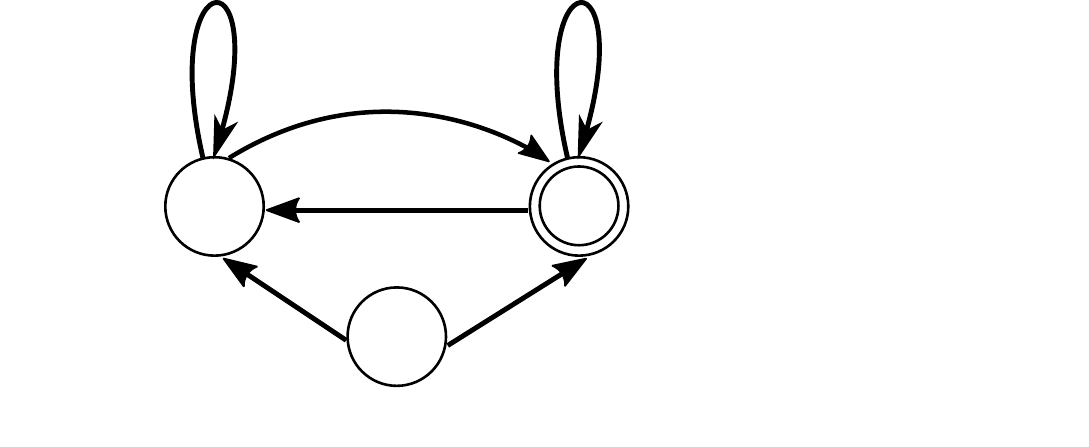
\caption{1-counter pushdown automaton accepting $\cL=\{w\in \{t^{-1},t\}^* \mid \sharp_t(w)\geq\sharp_{t^{-1}}(w)\}$.}
\label{fig:one-counter}
\end{figure}

We note that this language is not regular. 
This follows easily from the pumping lemma for regular languages (see \cite[Theorem 4.1]{HopUll}). 
However, the reader not familiar with the pumping lemma will find a proof that $\cL$ is not regular in  Lemma \ref{lem: coarsely non-decreasing}.
\end{ex}

\subsection{Closure properties of families of languages}
Let $X, Y$ be finite sets.
The set $X^*$ equipped with concatenation has the structure of a free monoid over the set $X$. 

Let $\cC$ be a class of languages.
We say that $\cC$ is {\it closed under concactenation} if for every $\cL, \cM\in \cC$ the language $\cL\cdot \cM \coloneqq \{u v \mid u\in \cL, \, v\in \cL\}$ belongs to $\cC$.
Given a language $\cL$, we denote by $\cL^n$ the concatenation of $\cL$ with itself $n$ times, i.e. $\cL^n=\cL \cdot \cL \cdots \cL$.
By $\cL^0$ we denote the language with just the empty word.
The {\it Kleene closure } of $\cL$ is the language $\cL^*\coloneqq \cup_{n\geq 0} \cL^n$.
We say that the class $\cC$ is {\it closed under  homomorphism} if for every $\cL\subseteq X^*$ with $\cL\in \cC$ and every monoid homomorphism $f\colon X^*\to Y^*$ we have that $f(\cL)\in \cC$.
Similarly, the class is {\it closed under inverse  homomorphism} if for every $\cL\subseteq X^*$ with $\cL\in \cC$ and every monoid homomorphism $f\colon Y^*\to X^*$ we have that preimage of $\cL$ under $f$,  here denoted $f^{-1}(\cL)$, belongs to $\cC$. 
Finally, the {\it reversal or mirror image} of $\cL\subseteq X^*$ is the language $\mathcal{L}^R := \{x_{n} x_{n-1} \dots x_1 \mid x_1 \dots x_n \in \mathcal{L}\}$.

A class of languages is called \emph{full abstract family of languages} (or {\it full AFL} for short) if it is closed under homomorphisms, inverse homomorphisms, intersections with regular languages, unions, concatenations and Kleene closure.

\begin{prop}\label{prop: AFL}
The classes $\Reg, \onecounter$ and $\CF$ languages are full AFL.

The classes $\Reg$ and $\CF$ languages are closed under reversal.

The class $\Reg$ is closed under complement, but $\CF$ is not.
\end{prop}

\begin{proof}
All the claims are contained in either \cite[Page 30]{Handbook}  for all classes but $\onecounter$. 
For $\onecounter$ the claims are contained in \cite[Chapter VII Theorem 4.4]{Berstel}.
\end{proof}

We  introduce a last type of machines that will  play an important role in the paper.

\begin{defn}\label{defn: rational transducer}
A {\it rational transducer} is a finite state automaton which can
output a finite number of symbols for each input symbol. 
Formally, a rational transducer is a six-tuple $\bT =(\cS,X,Y,\delta, s_0,\cA)$ where $\cS$, $X$, and $Y$ are finite sets called the states, input
alphabet, and output alphabet, respectively.
The function $\delta$ is a map from $\cS\times (X\sqcup \{\epsilon\})\to \textrm{Finite Subsets}(\cS \times Y^*)$.
The element $s_0\in \cS$ is the initial state and $\cA\subseteq \cS$ is the subset of accept states.
The interpretation of $(r, w) \in  \delta (s, x)$  is that $\bT$ in state $s$ with input symbol $x$
may, as one possible choice of move, enter the state $r$ and output the string $w\in Y^*$.

As usual, we extend recursively  $\delta$ to a function $\cS \times X^* \to \textrm{Finite Subsets}(\cS \times Y^*)$.
We have that $(r,v)\in \delta (s, ux)$ for some $x\in X$ and $u\in X^*$ if there is $(t,v')\in \delta(s,u)$ and $(r,w)\in \delta(t,x)$ such that $v=v'w$.
\end{defn}

\begin{rem}\label{rem: graph transducer}
We will also use a graph representation of a transducer $\bT$.
This is the graph associated to the underlying finite state automaton of $\bT$, with  labels on the edges indicating the output of reading a certain input. 
More precisely, if $\bT =(\cS,X,Y,\delta, s_0,\cA)$ is a rational transducer, we consider a graph with vertices $\cS$. There is a directed edge from $s\in \cS$ to $s'\in\cS$ if $(s',u)\in \delta(s,x)$ for some $x\in X\sqcup \{\epsilon\}$ and $u\in Y^*$. In that case, we put a label $x/u$ on that edge. It means that we can move along that edge from $s$  to $s'$ if we have input $x$ and on doing so we output the word $u$.
An example of a graphical representation can be found in Figure \ref{fig: transducer-f2}.
\end{rem}

For $u\in X^*$ let $\bT(u)\coloneqq\{v\in Y^* \mid (a,v) \in \delta(s_0,u) \text{ and } a\in \cA \}$.
With this, we can define {\it the image of $\cL\subseteq X^*$ under $\bT$} as  $\bT(\cL)\coloneqq\{\bT(u)\mid u\in \cL\}$. 
Finally, for $\cL\subseteq Y^*$, we can define the {\it inverse image of $\cL$ under $\bT$} as the set $\bT^{-1}(\cL)\coloneqq\{u\in X^* \mid (a,v)\in \delta(s_0, u) \text{ with } a\in \cA \text{ and }v\in \cL\}$.

\begin{prop}\label{prop: closure transducers}
Full AFL classes  are closed under rational transducers and inverse image under rational transducers.
\end{prop}
\begin{proof}
The inverse image of a language $\cL$ under a rational transducer $\bT$, is also the image of the language under some rational transducer $\bT'$  (See \cite[III.4 and III.Theorem 6.1]{Berstel}).
Nivat's theorem (see \cite[III.Theorem 4.1 in view of III.Theorem 6.1]{Berstel}) says that  if $\bT$ is a rational transducer with input alphabet $X$,  output alphabet $Y$ and $\cL\subseteq X^*$ is a language, then there is a finite alphabet $Z$, a regular language $\cR\subseteq Z^*$ and homomorphisms $f\colon Z^*\to X^*$ and $h\colon Z^*\to Y^*$ such that $\bT(\cL)=h(f^{-1}(\cL)\cap \cR)$.
Now it follows that full AFL classes are closed under rational transducers.
\end{proof}

\subsection{Subsets of groups described by languages}

Let $G$ be a group. 
A {\it (monoid) generating set for $G$ } is a set $X$ together with a surjective monoid homomorphism map $\pi\colon X^* \to G$.
If $\pi$ is known (for example if $X$ is a subset of $G$) we often just say that $X$ is a generating set.
We refer to images under $\pi$ as {\it evaluations}.

We will use $\ell(w)$ to denote the length of a word $w\in X^*$.
For an element $g\in G$, we set $|g|_X = \min \{\ell(w) \mid \pi(w)= g\}$.
For $g,h\in G$, the {\it word distance between $g$ and $h$}, is equal to $|g^{-1}h|_X$ and is denoted by $\dist_X(g,h)$.
We might drop the subscript $_X$ if the generating set is clear from the context.

The next lemma is standard and the proof is omitted. 

\begin{lem}\label{lem: independence gen set}
Let $\mathcal{C}$ be a class of languages closed under homomorphisms and inverse homomorphism. 
Let $(X,\pi_X)$ and $(Y,\pi_Y)$ be two finite generating sets of a group $G$.
Let $S\subseteq G$ be any subset.

\begin{enumerate}
\item There is a language $\cL_X\subseteq X^*$ in the class $\cC$ such that $\pi_X(\cL_X)=S$ if and only if there is a language $\cL_Y\subseteq Y^*$ in the class $\cC$ such that $\pi_Y(\cL_Y) = S$. 
\item There is a language $\cL_X\subseteq X^*$ in the class $\cC$ such that $\cL_X=\pi_X^{-1}(S)$ if and only if there is a language $\cL_Y\subseteq Y^*$ in the class $\cC$ such that $\cL_Y=\pi_Y^{-1}(S)$.
\end{enumerate}

\end{lem}

The following observation will be useful.
\begin{lem}\label{lem: negative cone in cC}
Let $(X,\pi)$ be a finite generating set of $G$ and $P\subset G$.
Let $\cC$ be a class of languages closed by homomorphisms, inverse homomorphisms and reversal. Then
\begin{enumerate}
\item [(i)]
there is a language $\cL\in \cC$ such that $\pi(\cL)=P$ if and only 
there is a language $\cL'\in \cC$ such that $\pi(\cL')=P^{-1}$, 

\item[(ii)]  
 $\pi^{-1}(P)\subseteq X^*$ is in the class $\cC$ if and only if 
 $\pi^{-1}(P^{-1})\subseteq X^*$ is in the class $\cC$.
\end{enumerate}
\end{lem}
\begin{proof}
Since $\cC$ is closed by homomorphisms and inverse homorphisms, Lemma \ref{lem: independence gen set} implies that the properties claimed in (i) and (ii) are independent of the generating set. 
So we will assume that $X\subseteq G$ is a finite generating set closed under taking inverses.

Let $f\colon X^* \to X^*$ be the map sending $x_1 \dots x_n \mapsto x_n^{-1} \dots x_1^{-1}$. 
Then $f$ is a composition of the homomorphism map sending $x \mapsto x^{-1}$ and the reversal map.
Moreover $f^2 = id \colon  X^* \to X^*$. 
In particular, $\cL\subseteq X^*$ is in $\cC$ if and only if $f(\cL)$ is in $\cC$.\\

Note that we have that $P^{-1}=\pi(f(\cL))$ and so (i) follows. 
Also $\pi^{-1}(P^{-1})=f^{-1}(\pi^{-1}(P))$ and (ii) follows.
\end{proof}

We will now recall a property that allows us to pass the complexity of  sets to subsets. 

\begin{defn}\label{def: L-convex}
Let $(X,\pi)$ be a finite generating set of $G$.
Let $\cL$ be a regular language over $X$.
A subset $H\subseteq G$ is  {\it language-convex with respect to $\cL$} (or {\it $\cL$-convex} for short) if there exists an $R\geq 0$ such that for each $w \in \cL$
with $\pi(w) \in H$, where $w= x_1\dots x_n$ and $w_i = x_1 \dots x_i$, all prefixes $w_i$ of $w$ satisfy that $\dist(\pi(w_i),\pi(\cL))\leq R$.
\end{defn}

The following result was proved  by the second author in \cite[Corollary 4.4]{Su}. 

\begin{prop}\label{prop: inherited by L-convex subgroups}
  Let $\cL$ be a regular language, let  $(X,\pi)$ be a finite  generating set of a group $G$ and let  $H$ be a subgroup of G. If $H$ is language-convex with respect to $\cL$, then there exists a regular language $\cL'\subseteq X^*$  that evaluates onto $H \cap \pi(\cL)$.
\end{prop}

\subsection{Left-orders and complexity classes}
A {\it left-order} $\prec$ on a group  $G$ is a total order on $G$ that is invariant under left $G$-multiplication.
That is, for all $a,b,c\in G$, $a\prec b$ holds if and only if $ca\prec cb$ holds.
If a left-order exists on $G$ we say that {\it $G$  is  left-orderable}. 
A {\it positive cone } of $G$ is a subsemigroup $P$ of $G$ with the property that $G$ is the disjoint union of $P$, $P^{-1}$ and $\{1_G\}$. 
A positive cone $P$ defines a left-order $\prec_P$ on $G$ by $a\prec_P b\Leftrightarrow a^{-1}b\in P$.
Conversely, a left-order $\prec$ on $G$ defines a positive cone $P_\prec=\{g\in G\mid 1_G\prec g\}$.

A group $G$ is {\it bi-orderable} if there is a left-order $\prec$ on $G$ such that it is also invariant under right $G$-multiplication.
In that case, $\prec$ is a {\it bi-order}.

We now link positive cones with formal languages.
\begin{defn}\label{defn: C-left-order}
Let $\mathcal{C}$ be a class of languages. 
Let $(X,\pi)$ be a  generating set of $G$.
Let  $\prec$ a left-order. We say that $\prec$ is a {\it $\mathcal{C}$-left-order}, or equivalently {\it $P_\prec$ is a $\mathcal{C}$-positive cone}, if there exists a language $\mathcal{L}\subseteq X^*$ in the class $\mathcal{C}$ such that $\pi(\cL) = P_\prec$.
\end{defn}

\begin{ex}\label{ex: cyclic}
An infinite cyclic group $G=\langle t\mid \; \rangle$ is left-orderable. 
Consider the generating set $\{t,t^{-1}\}$ of $G$.
The set $\{t^n  \mid n>0\}\subseteq G$ is a positive cone and it is easy to check that $\langle t \rangle^+=\{t^n  \mid n>0\}\subseteq \{t,t^{-1}\}^*$ is a regular language. 
Thus $\bZ$ admits $\Reg$-left-orders. 
\end{ex}

In Example \ref{ex:one-counter}, we saw that  $\pi^{-1}(\mathbb{Z}_{\geq 1})=\{w\in \{t,t^{-1}\}^* \mid \sharp_t( w )> \sharp_{t^{-1}} (w) \}$ is $\onecounter$ but not regular. Taking the full pre-image of a positive cone under $\pi$ gives a language that represents the positive cone but it might not be of minimal language complexity. This will be discussed in more detail in the Appendix.

We now get some applications from the lemmata we proved for the complexity of subsets in the previous subsection.
By Lemma \ref{lem: independence gen set}, we get the following result.

\begin{cor}
Let $\cC$ be a class of languages closed under homomorphisms (eg. $\cC=\Reg$ or $\cC=\CF$).
Admitting a $\cC$-left-order is a group property, independent of the generating set.
\end{cor}

Left-orders restrict to left-orders on subgroups. 
It is natural to ask if their complexity is preserved.
This is a subtle question. 
For example $F_2\times \mathbb{Z}$ has a $\Reg$-left-order  (Theorem \ref{thm: A*BxZ} or \cite{Su}), however $F_2$ does not have $\Reg$-left-orders \cite{HS}.
 
A  sufficient condition for a left-order to inherit the complexity class when passing to a subgroups is given in terms of language-convexity (see Definition \ref{def: L-convex}). For instance, since finite index subgroups of a finitely generated group are always language-convex, we get the following consequence of Proposition \ref{prop: inherited by L-convex subgroups} (which was stated as \cite[Theorem 1.1]{Su}).

\begin{prop} \label{prop: inheritance subgroups}
A $\Reg$-left-order on a finitely generated group $G$ restricts to a $\Reg$-left-order on any finite index subgroup.
\end{prop}

\subsection{Relative left-orders}\label{subsec: rel-ord}
We recall now some properties of relative positive cones. See \cite[Definition 4]{ADS} or \cite[Corollary 5.1.5]{Kopytov-Medvedev} for more details.
\begin{defn}\label{defn: rel-order} Let $G$ be a group and $H$ be a subgroup of $G$.
A subsemigroup $P\subseteq G$ is a {\it positive cone relative to $H$} if $G=P\sqcup H \sqcup P^{-1}$.
\end{defn}

Suppose that $P\subseteq G$ is a  positive cone relative to $H$. 
Then $P\cdot H \subseteq P$  holds.
Indeed, that  $P \cdot H$  has a non-trivial intersection with $P\inv \cup H$ in order to derive a contradiction. 
If $P \cdot H \cap P\inv \not= \emptyset$. 
Then there exists $p_1, p_2 \in P$ and $h\in H$ such that $p_1h = p_2\inv$, meaning that $h = p_1\inv p_2\inv$, so $H$ is not disjoint from $P\inv$, a contradiction. Similarly, if $P \cdot H \cap H$ is non-empty, then there exists $h_1, h_2 \in H$ and $p\in P$ such that $ph_1 = h_2$, meaning that $p = h_2 h_1\inv$, which implies  that $P$ and $H$ are not disjoint, a contradiction.  
A similar argument shows that $H\cdot P \subseteq P$.

With the previous remark, we see that if $P$ is a positive cone relative to $H$ then we can define a total $G$-invariant order on $G/H$ by setting $g_1H\prec g_2H\Leftrightarrow g_1^{-1}g_2\in P$. 
This is well-defined, since if we pick different coset representatitives such that $g_1'=g_1h$ and $g_2'=g_2h'$ then $g_1^{-1}g_2\in P$ implies that $h^{-1}g_1^{-1}g_2 h'\in H\cdot P \cdot H\subseteq  P$. 
The fact that $\prec$ is a total left-invariant order on $G/H$ follows now easily.

\begin{defn}
Let $\mathcal{C}$ be a class of languages. 
Let $(X,\pi)$ be a  generating set of $G$.
Let  $\prec$ a left-order. 
We say that $P\subseteq G$ is {\it $\mathcal{C}$-positive cone relative to $H\leqslant G$}, if $P$ is a positive cone relative to $H$ and there exists a language $\mathcal{L}\subseteq X^*$ in the class $\mathcal{C}$ such that $\pi(\cL) = P_\prec$.
\end{defn}

The following argument will be used several times:
\begin{lem}\label{lem: extend from convex}
Let $G$ be a group and $H$ a subgroup.
Let $P_{rel}$ be a positive cone relative to $H$ and $P_H$ a positive cone for $H$. Then $P=P_{rel}\cup P_H$ is a positive cone for $G$.

Moreover, if $\cC$ is a class of languages closed under union,  $P_{rel}$ is  $\cC$-positive cone relative to $H$ and $P_H$ is a $\cC$-positive  cone for $H$, then $P$ is a $\cC$-positive cone.
\end{lem}
\begin{proof}
We have that $G=P_{rel}\sqcup H\sqcup P_{rel}\inv=(P_{rel}\sqcup P_H) \sqcup \{1\} \sqcup (P_H^{-1}\sqcup P_{rel}\inv)=P\sqcup \{1\} \sqcup P^{-1}.$
To see that $P$ is a semigroup, note that we have $P_{rel} P_{rel} \subseteq P_{rel}$ and $P_H P_H \subseteq P_H$ by assumption, and that we observed previously that $P_{rel} H \subseteq P_{rel}$ and $H P_{rel} \subseteq P_{rel}$, therefore $(P_{rel} \cup P_H)(P_{rel} \cup P_H) \subseteq (P_{rel} \cup P_H)$.

The moreover part is clear by definition of $P$.
\end{proof}

\section{Lexicographic left-orders of group extensions}
\label{sec: extensions}

In this section we study left-orders on group extensions. Suppose that $N$ and $Q$ are groups and that $G$ is {\it an extension of $N$ by $Q$.}
That is, there is a short exact sequence $1\to N \to G \stackrel{f}{\to} Q \to 1$.
Fixing a right inverse  $s$ of $f$  (i.e. a section of $f$), with $s(1_Q)=1_G$, we obtain the bijection $G\to N \times Q$, $g\mapsto ( gs(f(g))^{-1}, f(g))$. In particular, it is natural to try to order $G$ out of this bijection with $N\times Q$.

\begin{defn}[Lexicographic order on direct products of totally ordered sets]
Let $(N,\prec_N)$ and $(Q,\prec_Q)$ be two totally ordered sets. \begin{itemize}
    \item  The {\it lexicographic order $\prec_{lex}$ on $N\times Q$ with leading factor $Q$} is given by $(n,q)\prec_{lex} (n',q')$ if and only if  $q\prec_Q q'$ or $q=q'$ and $n\prec_N n'$.
    \item The {\it lexicographic order $\prec_{lex}$ on $N\times Q$ with leading factor $N$} is given by $(n,q)\prec_{lex} (n',q')$ if and only if $n\prec_N n'$ or $n=n'$ and $q\prec_Q q'$.

\end{itemize}
\end{defn}

We will see that the lexicographic order on $N\times Q$ where $Q$ leads, always induces a left-order on the underlying group $G$ (Lemma \ref{lem: quotient-leads}). This is not the case for the lexicographic order where $N$ leads, yet in Section \ref{sec: kernel}, we will find conditions on orders  on $N$ and $Q$ (and the structure of $G$) so that the lexicographic order where $N$ leads induces a left-order on the group $G$.

\subsection{Lexicographic left-orders where the quotient leads}

\begin{lem}\label{lem: quotient-leads}
Let $f\colon G\to Q$ be a group epimorpishm  with kernel $N$.
Let $P_Q$ and $P_N$ be positive cones on $Q$ and $N$ respectively.
Then  $P_{\prec_{lex}}\coloneqq f^{-1}(P_{Q})\cup  P_{N}$ is a positive cone of $G$.
The left-order on $G$ induced by $P_{\prec_{lex}}$ is called \emph{ lexicographic led by the quotient $Q$}.
\end{lem}
\begin{proof}
Clearly $f^{-1}(P_Q)$ is a subsemigroup as it is the pre-image of a semi-group under a homomorphism. 
Also, as $Q=P^{-1}_Q\sqcup \{1_Q\}\sqcup P_Q$ we get that $G= f^{-1}(P_Q)^{-1}\sqcup N \sqcup f^{-1}(P_Q)$ and hence $f^{-1}(P_Q)$ is a positive cone relative to $N$.
The result now follows from Lemma \ref{lem: extend from convex}. 
\end{proof}

Let $(N,\prec_N)$ and $(Q,\prec_Q)$ be left-ordered groups, and let $P_N$ and $P_Q$ be the corresponding positive cones.
It is easy to check that indeed, the left-order of the previous lemma on $G$ coincides to the lexicographic order on the set $N\times Q$ (which is naturally in bijection with $G$) with leading factor $Q$.

\begin{prop}\label{prop:extensions}
Let $\mathcal{C}$ be a class of languages closed under unions
and inverse homomorphisms.
Let $N$ and $Q$ be finitely generated groups and $G$ an extension of $N$ by $Q$.
Let $P_N$ and $P_Q$ be $\cC$-positive cones for $N$ and $Q$ respectively.
Then $P_{\prec_{lex}}$ constructed as in Lemma \ref{lem: quotient-leads} is a $\cC$-positive cone for $G$.

In particular, $\cC$ can be any class of  full AFL languages as the classes $\Reg, \onecounter$ or $\CF$.
\end{prop}
\begin{proof}
Fix finite generating sets  $(X,\pi_N)$ and $(Y,\pi_Q)$ for $N$ and $Q$.
Let $P_N$ and $P_Q$ be  $\cC$-left-positive cones for $N$ and $Q$,
 and let $\cL_N\subseteq X^*$ and $\cL_Q\subseteq Y^*$ be in $\cC$ such that $\pi_N(\cL_N)=P_N$ and $\pi_Q(\cL_Q)=P_Q$.

Denote by $f$ the epimorphism of $G$ onto $Q$, i.e. $f\colon G\to Q$.
We can define a generating set  $(X\sqcup Y, \pi)$ for $G$ such that $\pi_N(x)=\pi(x)$ for $x\in X$ and $\pi_Q(y)=f(\pi(y))$ for $y\in Y$.

Let $\widetilde{\cL_Q}$ be the preimage of $\cL_Q$ under the monoid morphism $\pi_{X\sqcup Y \to Y}\colon (X\sqcup Y)^*\to Y^*$ that is the identity on $Y$ and sends elements of $X$ to the empty word. 
Note that $f(\pi(\widetilde{\cL_Q}))=P_Q$ and hence $\pi(\widetilde{\cL_Q})\subseteq f^{-1}(P_Q)$.

To see that $\pi(\widetilde{\cL_Q})\supseteq f^{-1}(P_Q)$, let $g\in G$ such that $f(g)\in P_Q$. 
Then, there exists $w\in \cL_Q$ such that $\pi_Q(w)=f(\pi(w))$. 
There is $\tilde{w}\in \widetilde{\cL_Q}$ such that $\pi_{X\sqcup Y \to Y}(\tilde{w})=w$ and therefore
if $\tilde{g}=\pi(\tilde{w})$ we get that $f(g)=f(\tilde{g})$ and hence $g(\tilde{g})^{-1}\in N$.
There is $u\in X^*$ such that $\pi(u)=\pi_N(u)=g(\tilde{g})^{-1}$. 
Note that $\pi_{X\sqcup Y \to Y}(u\tilde{w})=w$, therefore  $u\tilde{w}\in \widetilde{\cL_Q}$ and $\pi(u\tilde{w})=g(\tilde{g})^{-1}\tilde{g}=g$.

As $\cC$ is closed under inverse homomomorphism and $\widetilde{\cL_Q}\in \cC$, we get that $f^{-1}(P_Q)$ is $\cC$-positive cone relative to $N$.
By Lemma \ref{lem: extend from convex}, $P_G = f^{-1}(P_Q)\cup P_N$ is a $\cC$-positive cone.
\end{proof}

\subsubsection{Polycyclic groups}

We use Proposition \ref{prop:extensions} to show that left-orderable virtually polycyclic groups admit $\Reg$-left-orders. 
Let us recall some definitions.
\begin{defn}
Let $G$ be a group. 
A {\it subnormal series for $G$} is a decreasing sequence of subgroups of $G$
$$G = G_0 \unrhd G_1 \unrhd \dots  \unrhd G_n = \{1\}$$
such that $G_{i+1}$ is normal in $G_i$ for $0 \leq i < n$. 
The quotients $G_i/G_{i+1}$ are called {\it factors}.
\end{defn}

\begin{defn}
A group $G$ is {\it polycyclic} (resp. {\it poly-$\bZ$})  if there is a finite subnormal series for $G$ with  cyclic (resp. infinite cyclic)  factors.
\end{defn}

As a consequence of Proposition \ref{prop:extensions} we have the following.

\begin{cor}
Poly-$\bZ$ groups have $\Reg$-left-orders.
\end{cor}
\begin{proof}
The corollary follows by induction on the length of the subnormal series.
The base case is $\mathbb{Z}$ and Example \ref{ex: cyclic} shows that $\bZ$ has $\Reg$ positive cones. Proposition \ref{prop:extensions} allows the inductive argument. 
\end{proof}

Finally, we use a theorem of Morris \cite{Morris} to show the following.

\begin{prop}\label{prop: virtually polycyclic}
Left-orderable virtually polycyclic groups are poly-$\bZ$. In particular, they have $\Reg$-left-orders. 
\end{prop}
\begin{proof}
Recall that if $G$ is polycyclic, the number of  $\bZ$-factors in a subnormal series of cyclic factors is well-defined and called the {\it Hirsch length} and denoted $h(G)$.
For a  virtually polycyclic group $G$, we can define $h(G)$ to be the Hirsch length of any finite index polycyclic subgroup, and it is still well-defined (see for example \cite{Segal}). 
Moreover, if $H$ is a normal subgroup in $G$ then both $H$ and $G/H$ are virtually polycyclic and $h(G) = h(H) + h(G/H)$.

We argue by induction on the Hirsch length that a virtually polycyclic left-orderable group is poly-$\bZ$.
If $h(G)=0$, then $G$ left-orderable and finite, so $G$ is trivial.
Suppose that $h(G)>0$, and recall that Morris' theorem says that finitely generated, left-orderable amenable groups have infinite abelianization \cite{Morris}. 
Thus,  there is a normal subgroup $N\unlhd G$ such that $G/N\cong \bZ$.
Since $N$ is virtually polycyclic, left-orderable and $h(N)=h(G)-1$, we get by hypothesis that $N$ is poly-$\bZ$, and so is $G$.
\end{proof}

\subsubsection{Solvable Baumslag-Solitar groups}\label{sec: BS}
The result about regularity of  left-orders on polycyclic groups cannot be promoted to the case of solvable groups \cite{Dar}. 
Here we give the complete picture for when a solvable Baumslag-Solitar groups $BS(1,q)$, $q\in \bZ$ admits a regular positive cone.
Recall that for $q = 0$, $BS(1, 0)\cong \mathbb{Z}$ and for $q\neq 0$ these groups are defined by the presentation
$$BS(1,q)= \langle a, b \,|\, aba^{-1} = b^{q}\rangle.$$

For $q\neq 0$, it is well known that $BS(1,q)\simeq \bZ[1/q]\rtimes \bZ$ (see for instance \cite{Farb Mosher}).
Under this isomorphism,
the element $b$ can be identified with the unity of the ring $\bZ[1/q]$ (the minimal sub-ring of $\bQ$ containing $\bZ$ and $1/q$).
The element $a$ can be identified with  the generator for the $\bZ$-factor which acts on $\bZ[1/q]$ by multiplying by $1/q$. 
 Therefore, an element $r/q^{s}\in \bZ[1/q]$ where $s\geq 0$ and $r\in \bZ-\{0\}$, can be written as $a^{-s}b^{\varepsilon r}a^{s}$
where $\varepsilon\in \{-1,1\}$ depends  on the parity of $s$ and the sign of $r$ and $q$ as in the following table.
\begin{center}
\begin{tabular}{|c|c|c|}
\hline
 $\text{sign}(q)$& $s$   & $\varepsilon$ \\\hline
$+$ & even/odd & $\text{sign}(r)$ \\\hline
$-$ & even & $\text{sign}(r)$\\\hline
$-$ & odd & $-\text{sign}(r)$\\\hline

\end{tabular}
\end{center}
In particular, for solvable Baumslag-Solitar groups, every group element can be written in the form $a^n(a^{-m}b^k a^{m})$ with $m\geq 0$, $n,k\in \bZ$.

Notice that the groups $\mathbb{Z}$ and $\mathbb{Z}[1/q]$ admit only two left-orders each.
Therefore, there are only four lexicographic left-orders on $BS(1,q)\cong \mathbb{Z}[1/q]\rtimes \mathbb{Z}$ where the quotient leads.
These four left-orders have the following positive cones,
$$P_1=\{(r/q^s, n)\in \bZ[1/q]\rtimes \bZ \mid n>0 \text{ or } (n=0 \text{ and } r/q^s>0)\},$$
$$P_2=\{(r/q^s, n)\in \bZ[1/q]\rtimes \bZ \mid n<0 \text{ or } (n=0 \text{ and } r/q^s>0)\},$$
$P_3 := P_1^{-1}$ and $P_4 := P_2^{-1}$. We first remark that these are one-counter positive cones.

\begin{prop}\label{prop: BS(1,-q)}
For each, $q\neq 0$ and each  $i=1,2,3,4$, there is a one-counter language over $\{a, a^{-1} ,b, b^{-1}\}^*$ evaluating onto the positive cone $P_i$ of $BS(1,q)$ described above.
\end{prop}
\begin{proof}
We only deal  with the case $q<1$ and $i=1,2$ and we indicate the needed modifications for $q>1$ at the end of the proof.
The cases $i=3,4$ are identical.
The following two languages
$$ \cL_1=\{a^n(a^{-m}b^k a^{m}) \mid n>0 \text{ or } [n=0 \text{ and } (m \text{ odd } \text{ and }  k<0) 
\text{ or } (m \text{ even } \text{ and }  k>0)]
\}$$
$$\cL_2=\{a^n(a^{-m}b^k a^{m}) \mid n<0 \text{ or } [n=0 \text{ and } (m \text{ odd } \text{ and }  k<0) 
\text{ or } (m \text{ even } \text{ and }  k>0)]
\}$$
where we assume $m \geq 0$ and $k \in \bZ$,
 evaluate to $P_1$ and $P_2$. This follows easily from our description of normal forms at the beginning of the subsection.

We see that $\cL_1 = \{a\}^+ \cdot  \cL_{quot} \cup \cL_{ker}$ and $\cL_2 =   \{a^{-1}\}^+ \cdot  \cL_{quot} \cup \cL_{ker}$ where
 
$$\cL_{quot} = \{a^{-m}b^k a^{m} \mid m\geq 0, k \in \bZ\}$$ and 
$$\cL_{ker}= \{a^{-m}b^k a^{m} \mid m \geq 0, k \in \bZ,   (m \text{ odd } \text{ and }  k<0) 
\text{ or } (m \text{ even } \text{ and }  k>0)
\}.$$
By the closure properties of the class $\onecounter$,  to see that $\cL_1$ and $\cL_2$ are one-counter languages, it is enough to see that $\cL_{quot}$ and  $\cL_{ker}$ are in $\onecounter$. 
Moreover, 
\begin{align*}
\cL_{ker} & = \cL_{quot}\cap (\{a^{-1}a^{-1}\}^*\cdot \{b\}^*\cdot \{a a\}^*) \\
&  \cup \cL_{quot}\cap (\{a^{-1}\} \cdot \{a^{-1}a^{-1}\}^*\cdot \{b^{-1}\}^*\cdot \{a a\}^*\cdot \{a\})
\end{align*}
Again, since the class $\onecounter$ is closed under union and intersection with regular languages, we see that it is enough to show that $\cL_{quot}$ is one-counter. 
We illustrate a pushdown automaton accepting $\cL_{quot}$ in Figure \ref{fig: lquot} and leave the rest of the proof of this case to the reader.

\begin{figure}[ht]
\begin{center}
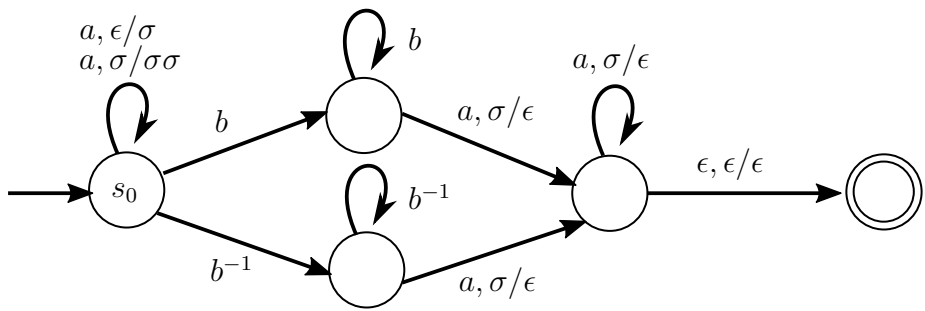
\end{center}
\caption{Pushdown automaton accepting $\cL_{quot}$. Here $\{\sigma\}$ is the  stack alphabet (i.e the counter symbol).}
\label{fig: lquot}
\end{figure}

For the case $q>1$, observe that 
$$ \cL_1'=\{a^n(a^{-m}b^k a^{m}) \mid n>0 \text{ or } [n=0 \text{ and }   k>0]
\}$$
$$\cL_2'=\{a^n(a^{-m}b^k a^{m}) \mid n<0 \text{ or } [n=0 \text{ and }   k<0]
\},$$
evaluate to $P_1$ and $P_2$ respectively. 
An easy modification of the previous automaton allows to conclude this case.
\end{proof}

We are going to show now that the previous proposition is optimal, in the sense that $P_1, P_2, P_3$ and $P_4$ can not be regular. To show this  we need to recall some facts.
\begin{defn}\label{defn: coarsely connnected}
Let $(\cM,\dist)$ be a metric space.  A subset $Y\subseteq \cM$ is {\it coarsely connected } if there is $R>0$ such that $\{p\in \cM\mid \dist(p,Y)\leq R\}$, the $R$-neighborhood of $Y$, is connected.
 \end{defn} 
The following is \cite[Proposition 7.2]{AABR}, in view of Lemma \ref{lem: negative cone in cC}.
\begin{lem}\label{lem: regular implies coarsely connected}
Let $G$ be a finitely generated group. If $P$ is a $\Reg$-positive cone of $G$, then $P$ and $P^{-1}$ are coarsely connected subsets of the Cayley graph of $G$.
\end{lem}

For Lemma \ref{lem: regular implies coarsely connected} recall that given a finitely generated group $G$, a non-trivial homomorphism $\phi\colon G\to \bR$ belongs to $\Sigma^1(G)$, the Bieri-Neumann-Strebel invariant (BNS invariant for short), if and only if  $\phi^{-1}((0,\infty))$ is coarsely connected. Moreover, the kernel of $\phi$ is finitely generated if and only if both $\phi$ and $-\phi $ belong to $\Sigma^1(G)$ (see \cite{BNS}).  
Therefore we have the following.
\begin{lem}\label{lem: BNS}
Suppose that $G$ is an extension of $N$ by $\bZ$. 
If $N$ is not finitely generated, then no lexicographic order on $G$ where $\bZ$ leads is $\Reg$.
\end{lem}
\begin{proof}
Suppose that $f\colon G \to \bZ$ is a homomorphism with kernel $N$. 
Without loss of generality, we can assume that a lexicographic order where $\bZ$ leads has a positive cone of the form $P_G\coloneqq f^{-1}(\bZ_{\geq 1})\cup P_N$
where $P_N$ is a positive cone for $N$. 

Since $f(P_N)=0$, and there is a group generator mapped under $f$ to a positive integer, we have that $f^{-1}(\bZ_{\geq 1})$ is coarsely connected if and only if $f^{-1}(\bZ_{\geq 1})\cup P_N$ is coarsely connected. Therefore, $f \in \Sigma^1(G)$ if and only if $f^{-1}(\bZ_{\geq 1})\cup P_N$ is coarsely connected. 
Similarly, $-f\in \Sigma^1(G)$ if and only if $f\inv(\bZ_{\leq -1})\cup P_N^{-1}$ is coarsely connected.

Assume that $N$ is not finitely generated and $P_G$ is a $\Reg$-positive cone. From BNS theory,  either $P_G=f^{-1}(\bZ_{\geq 1})\cup P_N$ or $P_G\inv=f\inv(\bZ_{\leq -1}) \cup P_ N\inv$ is not coarsely connected. By Lemma \ref{lem: regular implies coarsely connected} if the left-order given by $P_G$ is $\Reg$, then both cones $P_G$ and $P_G\inv$ would be coarsely connected. We have a contradiction.
\end{proof}

\begin{cor}\label{cor: BNS BS}
For each, $q\notin \{-1, 0, 1\}$ and each  $i=1,2,3,4$,  the positive cone $P_i$ of $BS(1,q)$ is not regular.
\end{cor}
\begin{proof}
It follows from Lemma \ref{lem: BNS} and the fact that $\bZ[1/q]$ is not finitely generated, that $P_1,P_2, P_3$ and $P_4$ can not be regular.
\end{proof}

Let $q>1$. From the point of view of orders, $BS(1,q)$ and $BS(1,-q)$ behave  quite differently.
 On one hand, $BS(1,-q)$ falls into the classification of groups admitting only finitely many left-orders obtained by Tararin \cite{Tararin} (see also \cite{Kopytov-Medvedev, GOD}, or Subsection \ref{sec: all left-orders are regular}). In fact, $B(1,-q)$ admits only the four left-orders we described before.) 
 On the other hand $BS(1,q)$ admits uncountably many left-orders. Beside the four left-orders described above, its other left-orders appears as induced orders from the (order-preserving) affine action of $BS(1,q)$ on the real line
 $$ a\colon x\mapsto qx\, , \; b\colon x\mapsto x+1.$$ 
 This classification of left-orders is shown in \cite{rivas jgt}  for the case of $BS(1,2)$ but the same proof works for general $q>1$.

We next observe that some of these induced orders have regular positive cones.
 
\begin{lem}\label{lem: BS(1,q)}
Let $BS(1,q)=\langle a,b \mid aba^{-1} = b^q\rangle$ with $q\geq 2$. 
Then $$P_0=\{ a^n(a^{-m}b^ka^{m})\mid k> 0, m,n\in \bZ\}$$ is a $\Reg$-positive cone relative to $\langle a \rangle$.
In particular, $P= \langle a\rangle^+ \cup P_0$ is $\Reg$-positive cone.
\end{lem}
\begin{proof}
Consider the action  of $BS(1,q)$ on $\mathbb{R}$ given by $a\colon x\mapsto qx$ and $b\colon x \mapsto x+1$. 
If $g\in BS(1,q)$ is an element with  normal form $g=a^{n}(a^{-m}b^ka^{m})$, then $g$ maps $x$ to $
q^nx+\frac{k}{q^m}$, which implies that this affine representation is faithful.

We let $P_0 = \{ g\in BS(1,q) : g(0)>0\}$.
 Note that an  element of the form $g=a^n(a^{-m}b^ka^m)$ belongs to $P_0$ if and only if $k>0$.
This coincides with the set $P_0$ of the statement, and therefore $P_0$ is a positive cone relative to $Stab_{BS(1,q)}(0)=\langle a \rangle \cong \bZ$.
From Lemma \ref{lem: extend from convex}, the set $P$  given by 
$$P=\{a^n\mid n>0\}\cup \{ a^n(a^{-m}b^ka^{m})\mid k> 0, m,n\in \bZ\},$$
is a positive cone for $BS(1,q)$.

Viewing the elements of $P$ as words in $\{a,b,a\inv,b\inv\}^*$, $P$ can be represented as a language accepted by the  automaton of Figure \ref{fig: BS(1,q)}. 
One can construct an automaton that gives a regular language for $P_0$ from the automaton of  Figure \ref{fig: BS(1,q)} by just changing the state right to $s_0$ to a non-accepting state.
\begin{figure}[ht]
\begin{center}
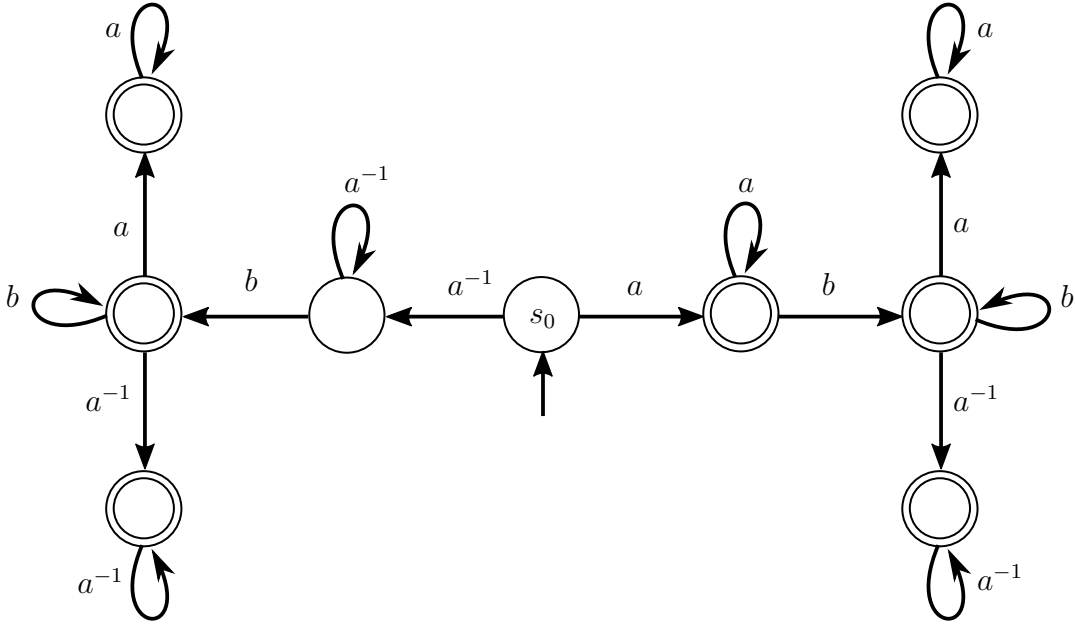
\end{center}
\caption{Finite state automaton accepting a positive cone language for $BS(1,q)$ with $q>1$.}
\label{fig: BS(1,q)}
\end{figure}

Remark that this is a subgraph of the automaton-graph describing the language of normal forms in $BS(1,q)$. 
\end{proof}

\begin{thm}\label{thm: BS}
 All solvable Baumslag-Solitar groups $BS(1,q)=\langle a,b\mid aba^{-1}=b^q\rangle$ admit one-counter left-orders, and $BS(1,q)$ admits a regular left-order  if and only if $-1\leq q$. 
Moreover, for $q\geq 2$ all $\Reg$-left-orders on $BS(1,q)$ are induced by affine actions on $\bR$.
\end{thm}

\begin{proof}
From Propositions \ref{prop: BS(1,-q)}  all solvable Baumslag-Solitar groups admit a one-counter left-order.

If $q\in \{-1,0,1\} $,  then $BS(1,q)$ is poly-$\bZ$ and then has regular left-orders (our convention is that $BS(1,0)=\bZ$).

If $q\leq -2$, by Tararin's classification discussed above and Corollary \ref{cor: BNS BS}, we have that  $BS(1,q)$ do not admit regular orders.

Finally, suppose that $q\geq 2$.
From \cite{rivas jgt},  a left-order on $G$ is either one induced from an affine action on $\bR$ or a lexicographic order with leading factor $\bZ$ coming from the semidirect product $\bZ[1/q]\rtimes \bZ$. 
We saw in Corollary \ref{cor: BNS BS} that the lexicographic left-orders on $BS(1,q)$ with leading factor $\bZ$ cannot be regular. 
We saw in Lemma \ref{lem: BS(1,q)} that $BS(1,q)$ admits regular orders and thus such left-order must be induced by affine actions (as it is shown in the proof of the lemma).
\end{proof}

\subsection{Lexicographic left-orders  where the kernel leads}\label{sec: kernel}

In general, if $G$ is an $N$ by $Q$ extension of left-orderable groups, the lexicographic order on $N\times Q$ with leading factor $N$ is not $G$-left-invariant. However, in some special situations it is. The following lemma will be used to produce left-orders on extensions where the kernel leads. 

\begin{lem}\label{lem: semidirect}
Let $G$ be a semidirect product of the subgroups $N \unlhd G$ and $Q\leqslant G$.
Let $P_N$ and $P_Q$ be positive cones of $N$ and $Q$ respectively, and assume that $qP_Nq^{-1} = P_N$ for all $q\in Q$.
Then, the lexicographic order on the underlying set $N\times Q$  with leading factor $N$ is $G$-left-invariant. In particular, it is a left-order on $G$.
\end{lem}
\begin{proof}
Since $G$ is a semidirect product of $N$ and $Q$, we have that $N\cap Q = \{1_G\}$ and $NQ=G$.
There is a natural bijection between $N\times Q$ to $G=NQ$, and under this bijection, the subset of elements of $N\times Q$ that are lexicographical greater to $(1_N,1_Q)$ corresponds to the subset $P\coloneqq P_N Q \cup P_Q$ of $G$.
Note that $P^{-1}= Q^{-1} P_N^{-1} \cup P_Q^{-1}$ and since $qP_Nq^{-1} = P_N$ for all $q\in Q$, we get that $P^{-1}= P_N^{-1} Q \cup P_Q^{-1}$. 
Therefore $G= P \sqcup P^{-1} \sqcup \{1_G\}$. 

To show that $P$ is a subsemigroup, let $nq, n'q'\in P$, with $n,n'\in N$ and $q,q'\in Q$. 
Recall that $n$ and $n'$ either are trivial or belong to $P_N$.
Then $nq n'q'= n (q n' q^{-1}) q q'$ and we see that if  $n \neq 1_G$ or $n' \neq 1_G$  then $ n (q n' q^{-1})\in P_N$ and $nq n'q'\in P_NQ\subseteq P$.
 If $n = 1_G = n'$, then $q,q'\in P_Q$ and $nq n'q'=qq'\in P_Q \subseteq P$.
\end{proof}

This lemma will be helpful to construct regular left-orders on extensions when the kernel is not finitely generated (and thus the kernel alone cannot support a regular left-order). Our main example are wreath products, however we have already observed this phenomenon in Baumslag-Solitar groups.

\begin{ex}[Lexicographic left-orders on $BS(1,q)$ where the kernel  leads]
For $q>0$ we have already seen that $BS(1,q)$ has regular orders, constructed through an affine action on the real line.
Viewing $BS(1,q)=\langle a,b \, \mid aba^{-1}=b^{q}\rangle \cong \bZ[1/q]\rtimes \bZ$, the left-order of Figure \ref{fig: BS(1,q)} is lexicographic where the factor $\bZ[1/q]$ leads.
Indeed, a positive cone $P_{\bZ[1/q]}$ for $\bZ[1/q]$ is the set of elements greater than $0$. 
Since $\langle a \rangle \cong \bZ$ acts on $\bZ[1/q]$ by multiplying by $q$, we have that $a^nP_{\bZ[1/q]}a^{-n}= P_{\bZ[1/q]}$ for all $n\in \bZ$. 
The previous lemma tell us that $\langle a \rangle \cdot P_{\bZ[1/q]} \cup \{ a \}^+$ is a positive cone for $BS(1,q)$.
 This is exactly the left-order of the automaton of Figure \ref{fig: BS(1,q)}.
 \end{ex}

\subsubsection{Wreath products}

The {\it wreath product} $N\wr Q$ of groups $N$ and $Q$ is the semidirect product of ${\bf N}:=\oplus_{q\in Q} N$ by $Q$, where the conjugation action of $Q$ on ${\bf N}$ is given by the left-multiplication action on the indexes of the copies of $N$.
That is, if ${\bf n}= (n_q)_{q\in Q}\in {\bf N}$ and $q'\in Q$, we have $q'{\bf n}q'^{-1}= (n_{q'q})_{q\in G}$.

\label{section wreath}
We begin with an example that will illustrate the construction we develop in this section.
\begin{prop} \label{prop: ZwrW}
The group $\bZ\wr \bZ$ has $\Reg$-left-orders.
\end{prop}
\begin{proof}
The group $\bZ\wr \bZ$ is isomorphic to $\bZ[X,X^{-1}]\rtimes C_\infty$ where  $C_\infty=\langle t\rangle$ acts by multiplying by $X$ on $\bZ[X,X^{-1}]$. 
An element of $\bZ\wr \bZ$ is therefore uniquely written as $(p=a_{i_0} X^{i_0} + a_{i_1} X^{i_1} + \dots a_{i_k} X^{i_k}, t^s)$  with $i_0 >i_1 >i_2 > \dots > i_k \in \bZ$ and $s\in \bZ$.
We say that  $a_{i_0}$ is the leading coefficient of $p$, and define $\mathsf{leadcoef}(p):=a_{i_0}$.
It is easy to check that  $P=\{(p,t^s) \mid  \mathsf{leadcoef}(p)>0 \}\cup \{(0,t^s) \mid s>0\}$ is a positive cone.

The group is generated by the constant polynomial $c=1$ and $t$, since $t^k c^m t^{-k}$ represents the polynomial $mX^k$.
Therefore a word representing the element given by $(a_{i_0} X^{i_0} + a_{i_1} X^{i_1} + \dots a_{i_k} X^{i_k}, t^s)$
is $t^{{i_0}} c^{a_{i_0}} t^{{i_1}-{i_0}} c^{a_{i_1}} t^{{i_2}-{i_1}} \cdots c^{a_{i_k}}t^{{-i_k}} t^s$.
Since $i_0 >i_1 >i_2> \dots >i_k$, we have that $i_1-i_0, i_2-i_1, \dots, i_{k}-i_{k-1}$ are all negative.
A language for this positive cone is 
$$\{t^n c^m t^{n_1} c^{m_1} t^{n_2} c^{m_2}\dots c^{m_k} t^l \mid m>0,  n_i<0, m_i\in \bZ, k \geq 0, l\in \bZ\} \cup \{t^s \mid s>0\}.$$
This language is recognized by the finite state automaton of Figure \ref{fig: wreath-fsa}.
\begin{figure}[ht] 
\begin{center}
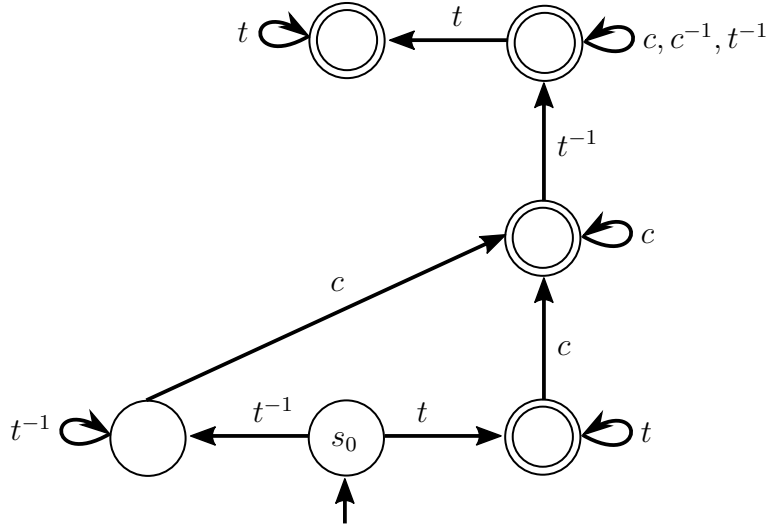
\end{center}
\caption{Finite state automaton accepting a positive cone language for $\bZ \wr \bZ$.}
\label{fig: wreath-fsa}
\end{figure}
\end{proof}

The previous strategy also works in the more general case of wreath products $N \wr Q$. Let $\prec_N$ and $\prec_Q$ be left-orders on $N$ and $Q$ respectively, with corresponding positive cones $P_N$ and $P_Q$.
We can construct a lexicographic order $\prec_{{\bf N}}$ on ${\bf N }= \oplus_{q\in Q} N$ as follows.  Given ${\bf n}=(n_q)_{q\in Q},{\bf n}'= (n'_q)_{q\in Q}\in {\bf N}$ we put ${\bf n}\prec {\bf n'}$ if ${\bf n}\neq {\bf n'}$ and for $q'=\max_{\prec_Q} \{q\in Q \mid n_{q}\neq n'_{q}\}$ we have that $n_{q'}\prec_N n'_{q'}$.

\begin{lem}\label{lem: lex on wreath}
The lexicographic order on $N \wr Q = {\bf N} \rtimes Q$ extending $\prec_{ \bf N}$ and $\prec_Q$ is an $N\wr Q$-left-invariant lexicographic order with leading factor ${\bf N}$.
\end{lem}

\begin{proof}
Observe that $P_{\bf N}$, the positive cone ${\bf N}$ associated to $\prec_{{\bf N}}$, is equal to $$\{ (n_q)_{q\in G} \in {\bf N}\setminus \{1_{\bf N}\} \mid n_q\in P_N \text{ where } q'=\max \{q\in G \mid n_q\neq 1_N\}\}.$$
Now, given ${\bf n}\in P_{\bf N}$ and $q'=\max \{q\in Q \mid n_q\neq 1_N\}$, and given $q\in Q$, we set ${\bf n}' = q {\bf n} q^{-1} = (n'_q)_{q\in Q}$, and we see that $qq' = \max_{\prec_Q} \{p\in Q \mid n'_p \neq 1_N\}$ since $\prec_Q$ is left $Q$-invariant.
Therefore $n'_{qq'}=n_{q'} \in P_N$ and thus ${\bf n}' \in P_{\bf N}$.
We have showed that $q P_{\bf N } q^{-1} \subseteq P_{\bf N}$ for all $q \in Q$, which implies  $q P_{\bf N} q^{-1} =P_{\bf N}$ for all $q \in Q$.
The proof then follows from Lemma \ref{lem: semidirect}. 
\end{proof}

Let $X$ and $Y$ be generating sets of $N$ and $Q$ respectively. 
The set $X\cup Y$ generates  $N\wr Q$ since the $q$-th copy of $N$ in $\oplus_{q\in Q}N$ is identified with $qNq^{-1}$ and thus it can be generated by $qXq^{-1}$ (the conjugates of $X$ by $q$) and each element of $qXq^{-1}$ can be expressed in terms of $X$ and $Y$. 

An element $({\bf n}=\{n_q\}_{q\in Q}, p)\in N\wr Q$ can be written as 
$(\prod_{q\in Q} q n_q q^{-1}) p$, and we can use the $\prec_Q$-order  to write this element uniquely as
$$(q_1 n_1 q_1^{-1})(q_2 n_2 q_2^{-1})\cdots (q_n n_m q_m^{-1}) p$$
with the property that $q_1\succ_Q q_2 \succ_Q q_3 \succ_Q \dots \succ_Q q_m$.

Thus, with this unique way of writing the elements of $N\wr Q$, the lexicographic positive cone of $N\wr Q$ is 
\begin{equation} \label{wr-pos}
P=\left\lbrace(q_1 n_1 q_1^{-1})(q_2 n_2 q_2^{-1})\cdots (q_m n_m q_m^{-1}) p \middle|  \begin{array}{c} q_1\succ_Q q_2 \succ_Q \dots \succ_Q q_m\\
(n_1 \in P_N \wedge  p\in Q) \vee (m=0 \wedge p\in P_Q) \end{array}\right\rbrace.
\end{equation}

We will now define a positive cone language for wreath products of groups in terms of the positive (and negative) cone languages for $N$ and $Q$.
\begin{prop} \label{prop: wr-lang}
Let $\cL_N\subseteq X^*$ and $\cL_Q, \cM_Q\subseteq Y^*$ be languages such that $\pi_N(\cL_N)=P_N$ and $\pi_Q(\cL_Q)=P_Q$ are positive cones for $N$ and $Q$ respectively, and  $\pi_Q(\cM_Q) = P_Q\inv$. 
Then, the  language
\begin{equation} \label{eq: pos-wr-AFL}
\cL \coloneqq Y^* \cL_N \cM_Q (X^*\cM_Q)^* Y^* \cup \cL_Q
\end{equation}
evaluates onto the positive cone $P$ of Equation \eqref{wr-pos}.  
\end{prop}
\begin{proof}
First observe that
\begin{equation} \label{eq: pos-wr}
\cL  =
 \left\lbrace v u_1 w_1 u_2 w_2 \dots u_m w_{m} z \middle| \begin{array}{c} v, z \in Y^*, \\ u_1 \in \cL_N \text{ or } (m=0, v= \varepsilon \text{ and } z \in \cL_Q), \\ u_i\in X^*, w_i \in \cM_Q\\
\end{array}\right\rbrace.
\end{equation}

Let $P$ be the positive cone described in \eqref{wr-pos}.

Let us first prove that $P \subseteq \pi(\cL)$. 
Let $g\in P$, and assume   that $g= (q_1 n_1 q_1^{-1})(q_2 n_2 q_2^{-1})\cdots \break (q_m n_m q_m^{-1}) p$ with $q_i,n_i,p$ as in \eqref{wr-pos}.
Let $z\in Y^*$ such that $\pi(z)=p$.
If $m=0$, then   $g=p \in P_Q$. 
We can assume that $z \in \cL_Q$.
If $m > 0$,  there is $v \in Y^*$ such that $\pi(v) = q_1$, $u_1 \in \cL_N$ such that $n_1 \in P_N$, and $w_i \in \cM_Q$ such that $\pi(w_i) = q_i\inv q_{i+1} \in P_Q\inv$, $u_i \in X^*$, such that $\pi(u_i) = n_i$ for $2 \leq i \leq m$. 
We see that $g= \pi(vu_1w_2u_2\dots u_mw_mz).$

To prove that $\pi(\cL) \subseteq P$, let ${\bf w}=v u_1 w_1 u_2 w_2 \dots u_m w_{m} z\in \cL$ as in the description in \eqref{eq: pos-wr}.
If $m=0$, then ${\bf w}= z$ and $z\in \cL_Q$. Thus $\pi({\bf w})\in P$.
If $m>0$, let $q_1=\pi(v)$ and for $i>1$, $q_i=q_{i-1}\pi(w_i)$, thus $\pi(w_i)=q_{i-1}^{-1}q_i$. For $i=1,\dots, m$ let $n_i=\pi(u_i)$ and $p=\pi(z)$.
Therefore $\pi({\bf w})=(q_1u_1q_1^{-1})(q_2u_2q_2^{-1})\cdots (q_mu_mq_m^{-1})p$.
Note that since $w_i\in \cM_Q$, we have that $q_i\prec_Q q_{i-1}$. It follows that $\pi({\bf w}) \in P$. 
\end{proof}

We now state a generalization of Proposition \ref{prop: ZwrW}.

\begin{prop}\label{prop:wreath product}
Let $\cC$ be a full AFL closed under reversal.
Let $N$ and $Q$ be finitely generated groups.
Suppose that $N$ and $Q$ have a $\cC$-left-order represented by $\cL_N$ and $\cL_Q$ respectively.
Then, the  wreath product  of $N\wr Q$ admits a $\cC$-left-order.

In particular admitting $\Reg$-left-orders is closed under wreath products.
\end{prop}
\begin{proof}
Assume that $N$ is generated by a finite set $X$ and $Q$ is generated by a finite set $Y$. 
We will construct a language over the generating set $X\sqcup Y$.
Let $\cM_Q = \cL_Q\inv$ be the  negative cone language associated to $\cL_Q$ obtained by reversal and sending each letter $x \mapsto x\inv$. 
Let $\cL$ be the language of Equation \eqref{eq: pos-wr-AFL}. 
By Proposition \ref{prop: wr-lang}, $\cL$ is a positive cone language for $N\wr Q$.
Since a class full AFL is  closed by concatenation, concatenation closure and union, we see that $\cL$ is in $\cC$.
\end{proof}

\section{Groups where all positive cones are regular}
\label{sec: all left-orders are regular}

In this section we classify the groups that only admit $\Reg$-left-orders.  

\subsection{Order-convex subgroups and language-convex subgroups}

\begin{defn}
Let $(G,\prec)$ be a left-ordered group.
A subgroup $H\leqslant G$ is called {\it $\prec$-convex} if for all $h_1,h_2\in H$ and $g\in G$ satisfying $h_1\prec g \prec h_2$ we have that $g\in H$.
\end{defn}

Left-order convexity relates to relative left-orders introduced in Section \ref{subsec: rel-ord} in the following way.

\begin{lem}
If $H$ is $\prec$-convex for a left-order $\prec$ on G, then $P = \{g \in G \mid H \prec g\}$ is a positive cone relative to $H$.
\end{lem}
\begin{proof}
First, we observe that for all $g \not\in H$ and all $h\in H$, we have that $g \prec h$ or $h \prec g$. 
Indeed, otherwise we would have that $h \prec g \prec h'$ for $h, h' \in H$, which by $\prec$-convexity of  $H$ implies that $g \in H$. 
Therefore, for all $g\in G- H$ either $g\prec H$ or $H\prec g$. 
Now, if $H \prec g$, then $g\inv \prec 1\in H$. 
As $g^{-1}\notin H$ this means  that $g \inv \prec  H$. 
This shows that $G = P \sqcup P\inv \sqcup H$. 
Finally, to show that $P$ is a semigroup, notice that since $1 \in H$, we have that $1 \prec g_1$ and $1 \prec g_2$ for $g_1, g_2 \in P$. 
This implies that $H \prec g_1 \prec g_1 g_2$. 
\end{proof}

Thanks to the previous lemma and the discussion in Section \ref{subsec: rel-ord}, we can conclude that if $H$ is $\prec$-convex, then for every  $g_1\prec g_2$,  $g_1,g_2\in G$ one has that  $g_1h_1\preceq g_2 h_2$ for all $h_1,h_2\in H$. Moreover, if $H$ is $\prec$-convex, then  $\prec$ induces a left-order on the coset space $G/H$. 

The proposition below says that for certain subgroups being $\prec$-convex implies being language-convex.  

\begin{prop}\label{prop: convex implies L-convex}
Let $G=H \rtimes \bZ$ be finitely generated by $(X,\pi)$. 
Let $\prec$ be a lexicographical  $\Reg$-left-order on $G$ led by $\bZ$, with $\cL\subseteq X^*$, a regular positive cone language.
If $H$ is finitely generated, then $H$ is language-convex with respect to $\cL$.

In particular, the restriction of $\prec$ to $H$ is a $\Reg$-left-order.
\end{prop}

Recall that we set in Notation \ref{not: counting letters}, that given word $w\in X^*$, and $x\in X$, we use $\sharp_x (w)$ to denote the number of times the letter $x$ appears in the word $w$.

Before proving the Proposition \ref{prop: convex implies L-convex} we need to describe all regular languages over the alphabet $\{t,t^{-1}\}$ mapping onto a  positive cone of $\bZ$.
Given a word $w=x_1\dots x_n\in \{t,t^{-1}\}^*$, with $x_i\in \{t,t^{-1}\}$ we define a function $f_w\colon \{0,1,\dots, n\}\to \bZ$ by $f_w(i)=\sharp_t(x_1\dots x_i)- \sharp_{t^{-1}}(x_1\dots x_i)$.

\begin{lem}\label{lem: coarsely non-decreasing}
Let $\cL \subseteq \{w\in \{t,t^{-1}\}^*\mid \sharp_t(w)- \sharp_{t^{-1}}(w)\geq 0\}$ be a regular language.
Then every $f_w$ is coarsely non-decreasing in the following sense: 
there is a constant $K\geq 0$ such that for all $w\in \cL$ and for all $i, j\in \{0,1\dots, \ell(w)\}$, if $j>i$ then  $f_w(j)>f_w(i)-K$.
\end{lem}
\begin{proof}
Let $\bM = (\cS, X=\{t,t^{-1}\}, \delta, s_0,\cA)$ be an automaton accepting $\cL$. 
We will consider $\bM$ to be without  $\epsilon$-moves and the image of transition function $\delta$ being singletons (see Remark \ref{rem: different types of fsa}).
We will think of this automaton as a directed graph as explained in Remark \ref{rem: fsa}.
Every edge has a label from $X$, and from every vertex there is at most one outgoing edge with label $x\in X$.
Hence, every $w\in \cL$ labels a unique a path in $\bM$ starting at the initial state $s_0$.

Every $w \in \cL$ can be decomposed  as $w = xyz$, where $y$ is a (possibly trivial) loop in $\bM$. 
Then, $w \in \cL$ implies that $xy^nz \in \cL$ for any $n\in \{0,1,2,\dots \}$. 
In other words, for each word $w$ accepted in the automaton for $\cL$, we may remove or insert words $y$ representing loops in $\bM$  and still get an accepted word $xy^nz$.

Let $w \in \cL$. Let $g(w) := \#_{t}(w) - \#_{t^{-1}}(w)$. Write $w = xyz$ where $y$ is a (possibly) trivial loop. 
Observe that $ g(y) \geq 0$, for otherwise $xy^iz \in \cL$ for any $i$, as we may pick $i$ large enough such that $g(xy^iz) < 0$, contradicting our assumption about $\cL$.
In other words, for any loop $y$ in $w=xyz$, 
\begin{equation}
\label{eq: y loop}
 g(xz)\leq g(xyz)
\end{equation}
and $xz \in \cL$ since we have only removed a loop $y$. 
 
Let $n$ be the number of states of $\bM$. 
For any subword $u$ of  $w \in \cL$, decompose $u = x_1 y_1 x_2 y_2 \dots x_{k-1} y_{k-1} x_{k}$, where the $y_i$ are loops and the length of the word $x_1x_2\dots x_k$ is minimal. 
We allow the subwords $x_i$ to be empty.
Viewing $u$ as a subpath of $w$ in $\bM$, we construct a new path whose label is $u' = x_1 \dots x_k$ which consists of removing all  loops from $u$.
 In particular, by the pigeonhole principle we have that $\ell(u') \leq n$ as otherwise the path associated to $u'$ would go through the same vertex in $\bM$ twice. Thus, we have a factorization of $u$ with a loop such that removing it produces a word shorter than $u'$. 
By \eqref{eq: y loop} we have that $g(u)\geq g(u')$, and as $\ell(u')\leq n$, this implies that 
\begin{equation}
\label{eq: g(u)}
g(u) \geq -n \text{ for all subword $u$ of $w\in \cL$,} 
\end{equation}
 since each transition can only contribute one $t$ or $t^{-1}$ and $\ell(u')\leq n$.

Let $1\leq i<j\leq \ell(w)$. 
We need to show that there is a constant $K\geq 0$ such that $f_w(j)-f_w(i)\geq -K$.
But $f_w(j)-f_w(i)=g(u)$ for $u$ equal to the subword of $w$ consisting of taking the prefix of $w$ of length $j$ and removing  to it the prefix of length $i$.
Now the result follows from \eqref{eq: g(u)} and taking $K=n$.
\end{proof}

Now we can prove Proposition \ref{prop: convex implies L-convex}.
\begin{proof}[Proof of Proposition \ref{prop: convex implies L-convex}]
Let $(X,\pi_H)$ be a generating set for $H$ and $\{t,t^{-1}\}$ a generating set for $\bZ$.
We combine them to make  $(X'=X\sqcup \{t,t^{-1}\}, \pi)$ a generating set for $G$ with evaluation map $\pi$.
Let $\cL\subseteq (X')^*$ be a regular language such that $\pi(\cL)$ is a lexicographic positive cone with the quotient being the leading factor.

Let $\phi \colon (X')^*\to \{t,t^{-1}\}^*$ consisting on deleting the letters of $X$.
This is monoid morphism, and hence $\phi(\cL)$ is regular.
Since $\cL$ is the language of a lexicographic order, $\phi(\cL)$ is contained in  $\{w\in \{t,t^{-1}\}^*\mid \sharp_t(w)- \sharp_{t^{-1}}(w)\geq 0\}$.
By Lemma \ref{lem: coarsely non-decreasing}, we get that there is a $K\geq 0$ such that  $f_{\phi(w)}(j)>f_{\phi(w)}(i)-K$ for all $w\in \cL$ and for all $j>i$.

To see that $H$ is language-convex with respect to $\cL$, let $w\in \cL$, with $\pi(w)\in H$.
Then, we get that $\sharp_t \phi(w)- \sharp_{t^{-1}}\phi(w)=0$.
Then, $0=f_{\phi(w)}(\ell(w))\geq \max_{i} f_{\phi(w)}(i)- K$, so $\max_{i} f_{\phi(w)}(i) \leq K$.
Also, $\min_{i} f_{\phi(w)}(i)> f_{\phi(w)}(0)-K=-K$.
It follows that for every prefix $u$ of $w$, $|\sharp_t \phi(u)- \sharp_{t^{-1}}\phi(u)|\leq K$.
Therefore $$\dist_G(\pi(u), \pi( ut^{-\sharp_t \phi(u)+ \sharp_{t^{-1}}\phi(u)}))\leq K.$$ Observe that $\pi( ut^{-\sharp_t \phi(u)+ \sharp_{t^{-1}}\phi(u)})\in H$ since the exponent ${-\sharp_t \phi(u)+ \sharp_{t^{-1}}\phi(u)}$ cancels the $t$'s in $u$. This shows that $H$ is language-convex.

By Proposition \ref{prop: inherited by L-convex subgroups}, we get that the restriction of the left-order to $H$ is regular.
\end{proof}

\subsection{Groups whose left-orders are all regular}

Now we will characterize the groups whose left-orders are all regular.  Since  the set of finite state automatons is countable,  a left-orderable group can have at most a countable number of regular left-orders.
The following result of Linnell \cite{Linnell} implies that if all the left-orders are regular, then there should be finitely many of them. 

\begin{thm}\cite{Linnell}
If a group admits infinitely many left-orders, then it admits uncountably many.
\end{thm}

The case when a group admits finitely many left-orders was classified by Tararin \cite{Tararin} (see also \cite{Kopytov-Medvedev, GOD}). Recall that a torsion-free abelian group has {\it  rank 1} if for any two non-identity elements $a$ and $b$ there is a non-trivial relation between them over the integers: $na+mb=0$.  Torsion-free abelian groups of rank $1$ are, up to isomorphism, subgroups of $\mathbb{Q}$.

A group $G$ is a {\it Tararin group} if it admits a unique subnormal series $$G = G_0 \unrhd G_1 \unrhd \dots  \unrhd G_n = \{1\}$$ such that all the factors are torsion-free abelian groups of rank 1 and such that no quotient $G_i/G_{i+2}$ is bi-orderable.

A group admits finitely many left-orders if and only if it is Tararin group (see \cite[Theorem 2.2.13]{GOD}).
Moreover, for any left-order $\prec$ on  a Tararin group with the subnormal series as above, the unique proper $\prec$-convex subgroups are the subgroups $G_1,G_2,\dots, G_n$. It follows that there are $2^n$ left-orders on $G$ and each of them is completely determined by choosing the positivity of a non-trivial element of $G_i/G_{i+1}$ for $i=1,\dots, n$.
More concretely, all left-orders on $G_i$ are lexicographic orders associated to the extension $G_{i+1}\to G_i \to G_i/G_{i+1}$ where the quotient leads.

\begin{lem}\label{lem: fg-Tararin-reg}
Suppose that $G$ is a finitely generated Tararin group with all the left-orders being regular. 
Then $G$ is is poly-$\bZ$.
\end{lem}
\begin{proof}
Let $$G = G_0 \unrhd G_1 \unrhd \dots  \unrhd G_n = \{1\}$$ be the unique subnormal series of $G$ where  all the factors are torsion-free abelian groups of rank 1.
We have to show that all factors are cyclic.
We argue by induction on the length of the series. 
If the length is $0$, $G\cong \{1\}$.

Suppose that the length is $>0$.
Since $G$ is finitely generated, we get that $G_0/G_1$ is a finitely generated subgroup of $\bQ$ and hence $G_0/G_1\cong \bZ$.
Thus, $G=G_1\rtimes \bZ$.

Suppose that $G_1$ is finitely generated. 
Then by Proposition \ref{prop: convex implies L-convex} all the induced left-orders on $G_1$ are regular.
Since $G$ is a Tararin group,  all left-orders on $G_1$ are restrictions of left-orders in $G$.
Therefore, all left-orders in $G_1$ are regular, and by induction $G_1$ is poly-$\bZ$ and so is  $G$.

The remaining case is that $G_1$ is not finitely generated. 
Then, it follows from Lemma \ref{lem: BNS} that the lexicographic orders cannot be regular.
\end{proof}

Conversely we have the following. 

\begin{lem}\label{lem: poly-Z-Tararin-reg}
All left-orders on a poly-$\bZ$ Tararin group are regular.
\end{lem}
\begin{proof}
The proof is by induction on the Hirsch length. If $h(G)=0$, then $G\cong \{1\}$ and the lemma holds.

Now assume that $h(G)>0$.
 Let $G_1\unlhd G$ such that $G/G_1$ is infinite cyclic. 
Since $G$ is a Tararin group $G_1$ is $\prec$-convex in $G$ for any left-order $\prec$. 
Thus, any order on $G$ is a lexicographic order associated to an extension $G_1$-by-$\bZ$ in which the quotient group is the leading lexicographic factor. 
Since $G_1$ is a poly-$\bZ$ Tararin group with $h(G_1)<h(G)$, we get by induction 
that all the left-orders on $G_1$  are regular.
Recall from Example \ref{ex: cyclic} that the two left-orders that $\bZ$ admits are regular.
Therefore,  all the left-orders of $G$ are lexicographic extensions of a regular order on $G_1$ and a regular order on $\bZ$ and by Proposition \ref{prop:extensions} all left-orders of $G$ are regular.
\end{proof}

\begin{rem}
Although we will not use this, it is worth pointing out that a group $G$ is poly-$\bZ$ Tararin if and only if there exists a unique subnormal series $$G = G_0 \unrhd G_1 \unrhd \dots  \unrhd G_n = \{1\}$$ 
such that for all $i$,  $G_i/G_{i+1}\cong \bZ$ and $G_i/G_{i+2}\cong K$ where $K=\langle a,b \mid aba^{-1}=b^{-1}\rangle$ is the Klein bottle group.
\end{rem}

\begin{thm}\label{thm:tararin}
A group  only admits regular left-orders if and only if  it is  Tararin poly-$\bZ$. 
\end{thm}
\begin{proof}
By the previous discussion a group that admits only regular left-orders must admit only a countable number of left-orders and therefore it must be a Tararin group.

To finish the proof, we need to show that a Tararin group only admits regular left-orders if and only if it is poly-$\bZ$. That result follows from Lemma \ref{lem: fg-Tararin-reg} and Lemma \ref{lem: poly-Z-Tararin-reg}.
\end{proof}

\section{Ordering quasi-morphisms}
\label{sec: quasimorphism}
In this section, we prove one of the main theorems of the paper and discuss some type of one-counter positive cones.
These one-counter positive cones will be constructed through a quasi-morphism, that we will call \emph{ordering quasi-morphism}, that is computable with a transducer.
The interesting fact of these one-counter positive cones, is that if $G$ has one of such cones, then $G\times \bZ$ will have a regular positive cone.
\subsection{Ordering quasi-morphism}

A {\it quasi-morphism} $\phi\colon G \to \bR$ is a function that is at bounded distance from a group homomorphism, i.e. there is a constant $D$ such that $|\phi(g)+\phi(h)-\phi(gh)|\leq D$.
A quasi-morphism  naturally assigns a positivity to elements of $G$. 
Dicks and \u{S}uni\'{c} \cite{DicksSunic2014}, found sufficient conditions on a quasi-morphism $\phi$ so that $\phi^{-1}( (0,\infty))$ is a positive cone in the group.

The following is a slight generalization of \cite[Lemma 2.3]{DicksSunic2014}.
\begin{lem}\label{lem: alt-poscone}
Let $G$ be a group and $\tau\colon G \to \mathbb{Z}$ be a function with the following properties. 
For all $g,h \in G$, 
\begin{enumerate}
	\item[(i)] $C=\{g\in G \mid \tau(g) = 0\}$ is a subgroup of $G$,
	\item[(ii)] $\tau(g) = -\tau(g\inv)$, 
	\item[(iii)] $\tau(g) + \tau(h) + \tau((gh)\inv) \leq 1$.
\end{enumerate}
Let $P = \{g \in G \mid \tau(g) > 0\}$. Then $P$ is a positive cone relative to $C$.
\end{lem}
\begin{proof}
We will first prove that $P\inv$ is disjoint from $P$. 
\begin{align*}
P\inv &= \{g \in G \mid g\inv \in  P \} \\
&= \{g \in G \mid \tau(g\inv) > 0\} \\
&= \{g \in G \mid -\tau(g) > 0\} \\
&= \{g \in G \mid \tau(g) < 0 \}. 
\end{align*}
Thus $P \cap P\inv = \emptyset$. Further, since $\tau(g) = 0 \implies g \in C$, we have that $G = P \sqcup P\inv \sqcup C$. 

Now we check that $P$ is a semigroup. Let $g,h \in P$. Then $\tau(g)\geq 1$ and $\tau(h)\geq 1$. Also, we have  that $\tau(g) + \tau(h) + \tau((gh)\inv) - 1 \leq 0$, so
\begin{align*}
\tau(gh) &\geq \tau(gh) + \tau(g) + \tau(h) + \tau((gh)\inv) - 1 \\
&= \tau(g) + \tau(h) -1 \geq 1.
\end{align*}
Thus $gh \in P$. 
\end{proof}
\begin{defn}
A quasi-morphism $\tau\colon G\to \bZ$ satisfying (i), (ii) and (iii) of Lemma \ref{lem: alt-poscone}, will be called an {\it ordering quasi-morphism}.
The subgroup $C$ of (i) is called the \emph{kernel} of $\tau$.
\end{defn}

\begin{ex}[Ordering quasi-morphism on free products]\label{ex: quasimorphism on  free product}
Dicks and \u{S}uni\'{c} proved in \cite[Proposition 4.2]{DicksSunic2014} that if $G=*_{i\in I} G_i$ is a free product of left-ordered groups $(G_i,\prec_i)$, then one can define an ordering quasi-morphism on $G$ as follows.

Let  $\prec_I$ be a total order on the index set $I$.  Given $g\in G$, the {\it normal form} of $g$ is an expression $g=g_1g_2\dots g_k$, where  each $g_i$ (called {\it syllable}) is in some $G_j\setminus \{1\}$ and two consecutive syllables $g_i,g_{i+1}$ lie in different free factors. 
The normal form of an element $g\in G$ is unique.
For a syllable $g_i$,  write $\textrm{factor}(g_i)=j$, to indicate that $g_i\in G_j$.
A syllable $g_i$ is positive if $1 \prec_{\textrm{factor}(g_i)} g_i$ and negative if $1 \succ_{\textrm{factor}(g_i)} g_i$.
An {\it index jump} in a normal form $g_1\dots g_k$ is a pair of  consecutive syllables $g_i g_{i+1}$ such that $\textrm{factor}(g_i)\prec_I \textrm{factor}(g_{i+1})$. 
Similarly, an {\it index drop} in a normal form $g_1\dots g_k$ is a pair of  consecutive syllables $g_i g_{i+1}$ such that $\textrm{factor}(g_i)\succ_I \textrm{factor}(g_{i+1})$.
Define 
\begin{align*}
\tau(g )&= \sharp(\text{positive syllables in $g$})- \sharp(\text{negative syllables in $g$})\\
&+ \sharp(\text{index jumps in $g$})- \sharp(\text{index drops in $g$}).
\end{align*}
Then \cite[Proposition 4.2]{DicksSunic2014} shows that the previous function is an ordering quasi-morphism on $G$ with trivial kernel.
\end{ex}

\begin{ex}[Ordering quasi-morphism on amalgamated free products]
\label{ex: quasimorphism on  amalgamated free product}
 In \cite[Theorem 17]{ADS}, the first author together with Dicks and \u{S}uni\'{c} described how to generalize the previous ordering quasi-morphism to more general groups acting on trees.
 We describe here the case of amalgamated free products. 
Suppose for example that $G$ is the free product of $G_i$, $i\in I$ amalgamated along a common subgroup $C\leqslant G_i$ for all $i$.  
For each $i$, assume that $P_i$ is a positive cone relative to $C$ (i.e. $G_i=P_i\sqcup C \sqcup P_i^{-1}$ and $P_i$ is a sub-semigroup of $G_i$).
Then, any element $g\in G$ can be written (uniquely fixing transversals)  as $g=g_1g_2\dots g_kc$ with $c\in C$ and $g_i\in P_{j_i}\cup P_{j_i}^{-1}$,  for $i\geq 0$.
Define
\begin{align*}
\tau(g )=\tau(g_1\dots g_k)&= \sharp(\text{positive syllables in $g_1\dots g_k$})- \sharp(\text{negative syllables in $g_1\dots g_k$})\\
&+ \sharp(\text{index jumps in $g_1\dots g_k$})- \sharp(\text{index drops in $g_1\dots g_k$}).
\end{align*}
Thus, from \cite[Theorem 17]{ADS},  $\tau$ is an ordering-quasimorphism on the free product of the $G_i$ amalagamated over $C$ with kernel $C$.
\end{ex}

\subsection{Ordering quasi-morphism computable through rational transducers}
We are interested in the situation when an ordering quasi-morphism  can be computed through a rational transducer.
\begin{defn}\label{defn: ordering quasi-morphism}
Let $G$ be a group finitely generated by $(X,\pi)$. Let $\bZ$ be generated by $(\{t^{-1},t\}, \pi_\bZ)$.
Let $\tau\colon G\to \bZ$ be an ordering quasi-morphism.

A {\it $\tau$-transducer} is a rational transducer $\bT$ with input alphabet $X$ and output alphabet $\{t^{-1},t\}$ such that
\begin{enumerate}
\item $G =\pi(\bT^{-1}(\{t^{-1},t\}^*))$, 
\item for every $w\in \bT^{-1}(\{t^{-1},t\}^*)$, one has that $\tau(\pi(w))= \pi_\bZ(\bT(w))$. 
\end{enumerate}
\end{defn}

In the following proposition we show that groups
admitting a $\tau$-transducer for an ordering quasi-morphism $\tau$ have one-counter positive cones.
\begin{prop}\label{prop: transducer imply one-counter}
Let $G$ be a group finitely generated by $(X,\pi)$ and $\tau \colon G\to \bZ$ an ordering quasi-morphism with kernel $C$.
If a $\tau$-transducer exists, then $P_\tau=\{g\in G \mid \tau(g)>0\}$ is a $\onecounter$-positive cone relative to $C$.

In particular, if there is a $\onecounter$ language $\cL_C$ such that $\pi(\cL_C)=P_C$ is a positive cone for $C$, then $P_\tau\cup P_C$ is a $\onecounter$ positive cone for $G$.
\end{prop}

\begin{proof}
By Lemma \ref{lem: alt-poscone}, $P_\tau$ is a positive cone relative to $C$.
The language $\cL=\{w\in \{t^{-1},t\} \mid \sharp_t(w)-\sharp_{t^{-1}}(w)>0\}$ is a one-counter language (See Example \ref{ex:one-counter}).
Let $\bT$ be a $\tau$-transducer.
Since the class of one-counter languages is closed under inverse image of rational transducers (Proposition \ref{prop: closure transducers} in view of Proposition \ref{prop: AFL}), we have that $\tilde{\cL}=\bT^{-1}(\cL)$ is a one-counter language.
Now, observe that from the Definition \ref{defn: ordering quasi-morphism}, we get that $\pi (\tilde{\cL}) = P_\tau$.

Finally, suppose that $\cL_C$ is a $\onecounter$ language such that $\pi(\cL_C)$ is a positive cone for $C$. Since one-counter languages are closed under union,  we get that $\tilde{\cL}\cup \cL_C$ is a one-counter language representing $P_\tau\cup P_C$.
\end{proof}

Let $\tau$ be an ordering quasi-morphism on a free product constructed as in Example \ref{ex: quasimorphism on  amalgamated free product}.
Our objective now is to construct $\tau$-transducer when the free factors have regular positive cones relative to $C$. 
The next lemma will be useful for this construction.

\begin{lem}
\label{lem: pm automaton}
Let $(X, \pi)$ be a finite generating set of $G$.
Suppose that $P$ is a $\Reg$-positive cone relative to $C\leqslant G$.

Then, there exists a non-deterministic finite state automaton  $\bM=(\cS,X, \delta, s_0,\cA)$ without $\epsilon$-moves where the set of accepting states $\cA$  is a disjoint union $\cA = \cA^- \sqcup \cA^+$ such that the language $\cL^-$ accepted by $\bM^-=(\cS,X,\delta,s_0,\cA^-)$, and the language $\cL^+$ accepted by $\bM^+=(\cS,X,\delta, s_0,\cA^+)$ satisfies that $\pi(\cL^-)=P^{-1}$ and $\pi(\cL^+)=P$.
\end{lem}
\begin{proof}
By hypothesis, there is a regular language $\cL^+=\cL\subseteq X^*$ that evaluates to $P$.
Moreover, by Lemma \ref{lem: negative cone in cC}, there is a regular language $\cL^-\subseteq X^*$ that evaluates to $P^{-1}$.
Now there are non-deterministic finite state automata $\bM^-$ and $\bM^+$ without $\epsilon$-moves accepting $\cL^-$ and $\cL^+$ respectively (Remark \ref{rem: different types of fsa}).
Viewing this automata as directed graphs, we obtain the desired automaton by identifying the start vertex of $\bM^-$ with the start vertex of $\bM^+$ and designing that vertex to be $s_0$, the start vertex of $\bM$. 
In Figure \ref{fig: mfsa} we see an example of this construction.
\end{proof}

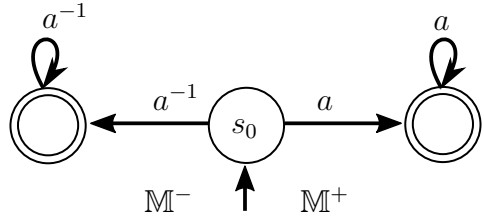
\begin{figure}[ht]
\begin{center}
\begingroup%
  \makeatletter%
  \providecommand\color[2][]{%
    \errmessage{(Inkscape) Color is used for the text in Inkscape, but the package 'color.sty' is not loaded}%
    \renewcommand\color[2][]{}%
  }%
  \providecommand\transparent[1]{%
    \errmessage{(Inkscape) Transparency is used (non-zero) for the text in Inkscape, but the package 'transparent.sty' is not loaded}%
    \renewcommand\transparent[1]{}%
  }%
  \providecommand\rotatebox[2]{#2}%
  \newcommand*\fsize{\dimexpr\f@size pt\relax}%
  \newcommand*\lineheight[1]{\fontsize{\fsize}{#1\fsize}\selectfont}%
  \ifx\svgwidth\undefined%
    \setlength{\unitlength}{179.35838114bp}%
    \ifx\svgscale\undefined%
      \relax%
    \else%
      \setlength{\unitlength}{\unitlength * \real{\svgscale}}%
    \fi%
  \else%
    \setlength{\unitlength}{\svgwidth}%
  \fi%
  \global\let\svgwidth\undefined%
  \global\let\svgscale\undefined%
  \makeatother%
  \begin{picture}(1,0.42942464)%
    \lineheight{1}%
    \setlength\tabcolsep{0pt}%
    \put(0,0){\includegraphics[width=\unitlength,page=1]{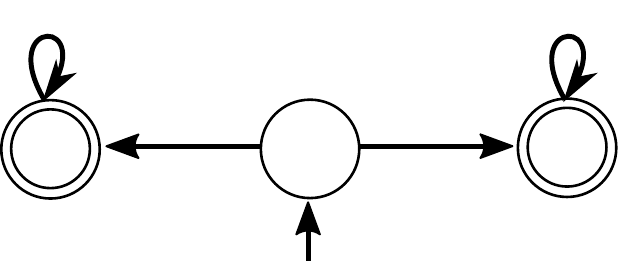}}%
    \put(0.46552185,0.17374266){\color[rgb]{0,0,0}\makebox(0,0)[lt]{\lineheight{1.25}\smash{\begin{tabular}[t]{l}$s_0$\end{tabular}}}}%
    \put(0.64592594,0.22372262){\color[rgb]{0,0,0}\makebox(0,0)[lt]{\lineheight{1.25}\smash{\begin{tabular}[t]{l}$a$\end{tabular}}}}%
    \put(0.89078028,0.39098549){\color[rgb]{0,0,0}\makebox(0,0)[lt]{\lineheight{1.25}\smash{\begin{tabular}[t]{l}$a$\end{tabular}}}}%
    \put(0.30117859,0.22372262){\color[rgb]{0,0,0}\makebox(0,0)[lt]{\lineheight{1.25}\smash{\begin{tabular}[t]{l}$a^{-1}$\end{tabular}}}}%
    \put(0.06853534,0.39510081){\color[rgb]{0,0,0}\makebox(0,0)[lt]{\lineheight{1.25}\smash{\begin{tabular}[t]{l}$a^{-1}$\end{tabular}}}}%
    \put(0.2835857,0.0085374){\color[rgb]{0,0,0}\makebox(0,0)[lt]{\lineheight{1.25}\smash{\begin{tabular}[t]{l}$\bM^{-}$\end{tabular}}}}%
    \put(0.61126803,0.00906529){\color[rgb]{0,0,0}\makebox(0,0)[lt]{\lineheight{1.25}\smash{\begin{tabular}[t]{l}$\bM^{+}$\end{tabular}}}}%
  \end{picture}%
\endgroup%

\end{center}
\caption{Example of the construction of Lemma \ref{lem: pm automaton} for $\bZ$, which gives an automaton that have states for recognizing the positive cone and states for recognizing the negative cone.}
\label{fig: mfsa}
\end{figure}

\begin{prop}\label{prop: free-prod-cf}
Let $I$ be a finite set.
For each $i\in I$, let $G_i$ be a group finitely generated by $(X_i,\pi_i)$ and assume that all $G_i$ have a common finitely generated subgroup $C$.
Let $G$ denote the free product of the $G_i$ amalgamated over $C$.

Let  $(Y,\pi_C)$ be a finite generating set for $C$ and $(X=Y\sqcup \bigsqcup X_i,\pi)$ be a generating set for $G$  where for $x\in X_i$ we have that $\pi(x)=\pi_i(x)$ and for $y\in Y$ we have that $\pi(y)=\pi_C(y)$.

For each $i\in I$, assume that $P_i$ is a positive cone relative to $C$.
Suppose that there is a regular language $\cL_i\subseteq X_i^*$ such that $\pi_i(\cL_i)=P_i$.

Then $G$ has an ordering quasi-morphism $\tau\colon G\to 2\bZ+1 \cup \{0\}$ with  kernel $C$ admitting a $\tau$-transducer $\bT$. 
Moreover, the following language  $$\bT^{-1}(\{w\in \{t,t^{-1}\}^*: \sharp_t(w)-\sharp_{t^{-1}}(w)>0\})$$ evaluates onto a positive cone relative to $C$ and it is equal to $$\{w = w_{i_1} \dots w_{i_m} z \in X^*\mid i_t\in \{1,\dots, n\},  i_j \not= i_{j+1},  w_{i_j} \in \cL^+_{i_j} \cup \cL^-_{i_j}, z \in Y^*, \tau(\pi(w)) > 0 \}.$$
\end{prop}

\begin{proof}
By the previous Lemma \ref{lem: pm automaton}, there are finite state automata $\bM_i =(\cS_i,X_i, \delta_i, s_{0i},\cA_i^-\sqcup \cA_i^+)$ for $i=1,\dots, n$, such that the words accepted by $\bM_i$ on a state from  $\cA_i^{+}$ form a regular language $\cL_i^+$ evaluating onto $P_i$ and the words accepted by $\bM_i$ on  a state from  $\cA_i^{-}$ form a regular language $\cL_i^-$ evaluating onto $P_i^{-1}.$

 We  first construct a non-deterministic finite state automaton $\bM$ accepting the language 
 $$\cL := \{w = w_{i_1} \dots w_{i_m} z \in X^*\mid i_t\in \{1,\dots, n\},  i_j \not= i_{j+1},  w_{i_j} \in \cL^+_{i_j} \cup \cL^-_{i_j}, z \in Y^*\}.$$
 Then, we will modify this automaton $\bM$ to produce a $\tau$-transducer.
 An example of this construction can be found in Figure \ref{fig: transducer-f2} and might help the reader check the example as they read the proof.

Note that $G- C$, the complement of $C$ in $G$, is equal to $\pi(\cL )$.
Note also that by Example \ref{ex: quasimorphism on  amalgamated free product}, $\{w\in \cL \cdot Y^* \mid \tau(\pi(w))>0\}$ evaluates onto a  positive cone relative to $C$ of $G$.

We construct $\bM$ taking the union of the automata $\bM_i$, $i=1,2,\dots, n$ with their transitions and we add an extra start state $s_0$ and a final state $f$ and the following $\epsilon$-moves:
\begin{enumerate}
\item[(I)] there is an $\epsilon$-move from $s_0$ to every start state each $\bM_i$.
\item[(II)] there is an $\epsilon$-move from each accepting states of $\bM_i$ to the start state of $\bM_j$ with $i\neq j$.
\item[(III)] there is an $\epsilon$-move from each accepting states of $\bM_i$ to $f$.
\end{enumerate}
The state $f$ is the only accepting state of $\bM$.
We add loops on $f$ with label $y$, for each $y\in Y$.
It is easy to see that $\cL$ is accepted by $\bM$.

Now we construct a $\tau$-transducer $\bT$ from $\bM$ by adding some outputs on $T=\{t^{-1},t\}$ to the $\epsilon$-moves.
It will be clear from the construction that  the output is a word in $t$ and $t^{-1}$ that evaluates into an odd number,  unless the input is a word in $Y$ in which case the output evaluates to 0.

The $\epsilon$-moves of type I do not output any word.

The $\epsilon$-moves of type II output $tt$ if they start on some $\cA_i^+$ $i=1,\dots, j$ and go to the start state of $\bM_j$ with $i<_Ij$.
The $\epsilon$-moves of type II output $t^{-1}t^{-1}$ if they start on some $\cA_i^-$ $i=1,\dots, j$ and go to the start state of $\bM_j$ with $i>_Ij$.
The other $\epsilon$-moves of type II do not  output any word.
Note that these encode the contributions of index jumps and drops, and positive syllables and negative syllable as defined in Example \ref{ex: quasimorphism on  amalgamated free product}.

The $\epsilon$-moves of type III output a $t$ if they start on some vertex of $\cA^+_i$ and outputs a $t^{-1}$ if they start on some vertex of $\cM^-_i$.

It is easy to see that this gives a $\tau$-transducer.
\end{proof}

\begin{figure}[ht]
\begin{center}
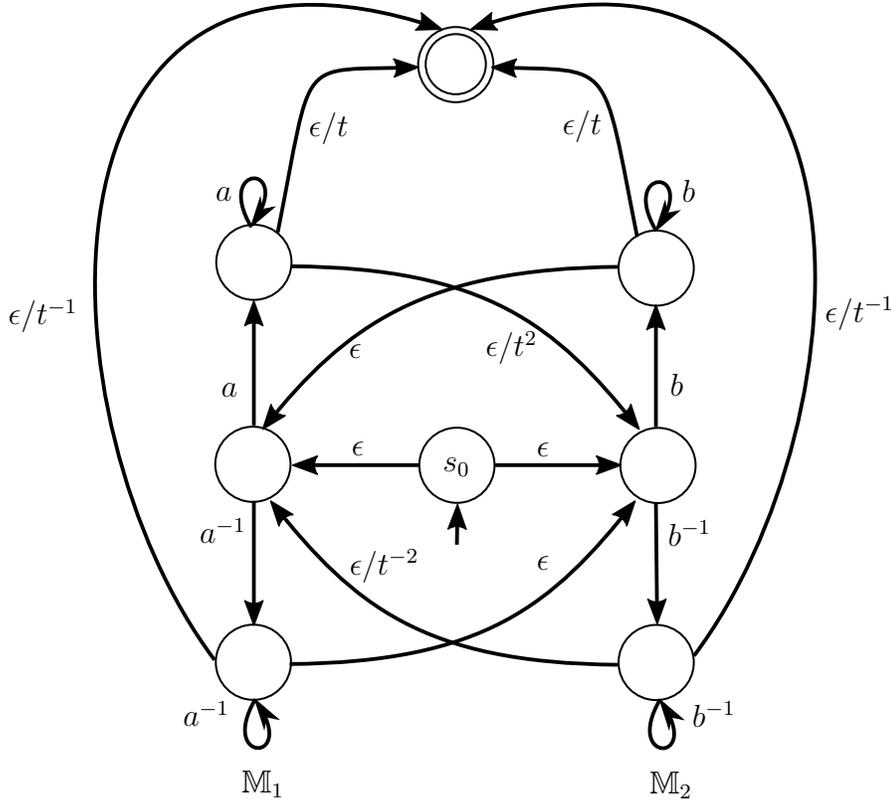
\end{center}
\caption{The $\tau$-transducer $\bT$ for $F_2 = \langle a \rangle * \langle b \rangle$ constructed following the proof of Proposition \ref{prop: free-prod-cf}. 
Observe that the three vertices on the left, and the three vertices of the right are copies of the automaton of Figure \ref{fig: mfsa}. 
Observe also that $\bT^{-1}(\{t,t^{-1}\}^*)$ consists of all reduced words in $\{a,b,a^{-1},b^{-1}\}$. }
\label{fig: transducer-f2}
\end{figure}

From Propositions \ref{prop: transducer imply one-counter} and \ref{prop: free-prod-cf}, we get the following corollary.

\begin{cor}\label{cor: free prod admit one-counter order}
Let $A, B$ be groups admitting $\Reg$-left-orders. Then $A*B$ admits a $\onecounter$-left-order.
\end{cor}

This corollary is optimal in the view of a result of Hermiller and \u{S}uni\'{c} \cite{HS} that says that free products do not admit regular positive cones.  
We also note that the first and second authors, together with J. Alonso and J. Brum \cite[Theorem 1.6]{AABR}, proved that certain free products with amalgamation also do not admit regular positive cones. 

To introduce an interesting example\footnote{Actions of $BS(1,m;1,n)$ on the closed interval $[0,1]$ have been studied in \cite{BMNR}, where it is showed that $BS(1,m;1,n)$ has no faithful action on $[0,1]$ by diffeomorphisms.} of an amalgamated free product allowing a $\onecounter$ positive cone, let $n,m$ be two positive integers and consider the group
 \begin{align*}
 BS(1,m;1,n)& \coloneqq \langle a, b, c \mid aba^{-1} = b^m,\, aca^{-1} = c^n \rangle \\
 & \cong \langle a, b \mid aba^{-1} = b^m \rangle *_{\langle a \rangle} \langle a, c \mid  aca^{-1}= c^n \rangle.
 \end{align*}
 
By Lemma \ref{lem: BS(1,q)}, $BS(1,n)$ has $\Reg$-positive cones relative to $\langle a \rangle$.
Therefore by Proposition \ref{prop: free-prod-cf}, $BS(1,m;1,n)$ admits a $\onecounter$-positive cone relative to $\langle a \rangle$ and by  Proposition \ref{prop: transducer imply one-counter} we get the following.
\begin{cor}\label{cor: BS(1,m;1,n)}
For $n,m\geq 1$, the group  $BS(1,m;1,n)$ has a $\onecounter$-positive cone. 
\end{cor}

\subsection{Embedding theorem}
The main result of this section is a construction of a regular left-order on $G\times \bZ$ starting from an ordering quasi-morphism $\tau\colon G\to \bZ$.
We begin  describing a positive cone for $G\times \bZ$, and then  show that it is regular (Theorem \ref{thm: key-embedding}).

\begin{prop} \label{prop: pos-cone-cross-z}
Let $G$ be group,  $C$ be a subgroup of $G$, and $\tau\colon G \to 2\bZ + 1 \cup \{0\}$ be an ordering quasi-morphism  with kernel $C$. 
Let $$ P = \{(g,n) \in G \times \bZ \mid \tau(g) + 2n > 0 \}.$$
 Then, the set $ P$ is a positive cone relative to $C\times\{0\}$ for $G \times \bZ$.
\end{prop}

\begin{proof}
We define a map $\tau'\colon G\times \bZ\to \bZ$ as  $\tau'((g,n))=\tau(g)+2n$ and show that it satisfies the conditions of Lemma \ref{lem: alt-poscone}.
For (i), since $\tau(g)$ is odd for every $g\in G-C$, we get that $\tau'((g,n))=0$ if and only if $g\in C$ and $n=0$.
For (ii), observe that $-\tau'((g,n))=-\tau(g)-2n= \tau(g^{-1})-2n =\tau'((g^{-1},-n))$ for all $(g,n)\in G\times \bZ$.
Finally, for (iii), observe that $\tau'((g,n))+\tau'((h,m))-\tau'((g^{-1}h^{-1},-n-m))=\tau(g)+\tau(h)-\tau( g^{-1}h^{-1})\leq 1$ for all $(g,n),(h,m)\in G\times \bZ$.
\end{proof}

Our main theorem is the following.

\begin{thm}\label{thm: key-embedding}
Let $G$ and $\tau$ be as in Proposition \ref{prop: pos-cone-cross-z}. If a $\tau$-transducer exists, then the set $P = \{(g,n) \in G \times \bZ \mid \tau(g) + 2n > 0 \}$ can be represented by a regular language.
\end{thm}

Note that it is implicit in the hypothesis of the existence of a $\tau$-transducer that $G$ is finitely generated.

\begin{rem}\label{rem: labelled edges}
In the proof of Theorem \ref{thm: key-embedding} we will make use of the graph descriptions of finite state automata (Remark \ref{rem: fsa}) and transducers (Remark \ref{rem: graph transducer}). 
In the case of finite state automata we will allow ourselves to label edges by words in the input alphabet and not just letters. To get a proper automaton, one has to change the edges labelled by a word $w$ by a path of the length $\ell(w)$ whose label is $w$. This will allow us to simplify the presentation of the proof. 
\end{rem}

\begin{proof}[Proof of Theorem \ref{thm: key-embedding}]
Let  $\tau'\colon G\times \bZ\to \bZ$ be defined by  $\tau'((g,n))=\tau(g)+2n$, so that $P=\{(g,n)\mid \tau'((g,n))>0\}$.
Let $(X,\pi)$ be a finite generating set for $G$, and
extend it to a generating set $(X \sqcup Z,\pi)$ of $G\times \bZ$,  where the elements of $Z = \{z,z\inv\}$ evaluate to $1$ and $-1$ on $\bZ$ respectively.
Let  $\bT =(\cS,X, T=\{t,t^{-1}\}, \delta_{\bT}, s_0, \cA)$ be a $\tau$-transducer.
We will modify $\bT$ to construct a finite state automaton  $\bM$ accepting a language ${\cL}$ that will evaluate onto ${P}$. 

We will construct $\bM$ to have a key property, called the \emph{balancing property}. To make the proof easier to follow, we have divided the steps of construction so that we will first define the features of the automaton unrelated to the balancing property. Then, we will define the balancing property and show why this property is important. Finally, we will finish the construction of the automaton such that the balancing property is respected. As a visual aid, we remark that Figure \ref{fig: reg-trans} exemplifies our construction starting from the transducer of Figure \ref{fig: transducer-f2}. The reader might find helpful to check these examples while following the construction.

We construct  $\bM$ as follows. The alphabet of the finite state automaton $\bM$ is $X\sqcup Z$.
 The set of states $\cS_\bM$ of $\bM$ is equal to $\{f\}\cup \left( \cS\times \{0,1\} \right) $ where $\cS$ are the states of $\bT$ and $f$ is a new state.
The initial state of $\bM$ is $(s_0,0)$.
The accepting states of $\bM$ are the states $\{(\alpha,1) \mid \alpha \in \cA\}\cup \{f\}$.

The transition function of $\bM$ is denoted by $\delta_\bM$. The transitions that go to $f$ are as follows. First, $\delta_\bM(f,z)=f$ (i.e. if a word is accepted, we can keep reading $z$'s). 
Second $\delta_\bM((s_0, 0), z)= f$ (i.e. all words of $\{z\}^+$ are accepted).
We have that $\delta_\bM ((\alpha,\dagger), z)= f$ for all $\dagger\in \{0,1\}$ and $\alpha \in \cA$. 

There are no more transitions from the state $f$ or going to the state $f$.

We will complete the construction of $\bM$ such that the following property holds.

\vspace{0,2cm}
\noindent {\bf Balancing property:}  For $w\in (X\sqcup Z)^*$, $s \in \mathcal S$ and $\dagger \in \{0,1\}$, we have that $(s,\dagger)\in \delta_\bM((s_0,0),w)$ implies that $\tau'(\pi(w))=\dagger$. 
\vspace{0,2cm}

Assuming the Balancing property we see that $\bM$ only accepts words that represent elements of $P$. Indeed, if a word $w$ is accepted at the accepting state $(\alpha,1)$ with $\alpha\in \cA$, then by the Balancing property we have that $\tau'(\pi(w))=1$, so $\pi(w)\in P$. On the other hand, if $w$ is accepted at the accepting state $\{f\}$, then, since all the edges in $\bM$ that end in the accepting state $\{f\}$ have label $z$, we can assume that $w$ is of the form $w'z^n$ where $w'$ does not end with $z$. Now, by the construction of transitions to $\{f\}$, we see that  $f\notin \delta_{\bM}((s_0,0),w')$.
Suppose that $(s,\dagger)\in \delta_{\bM}((s_0,0),w')$ for some $s\in \cS$. By the Balancing property we have that $\tau'(\pi(w'))=\dagger\in \{0,1\}$, so $\tau'(\pi(w))= \tau'(\pi(w'z^n))=\tau'(\pi(w'))+2n>0$. This implies that  $\pi(w)\in P$ as well, and therefore that $\pi(\cL) \subseteq P$.

We now define the other transitions so that the Balancing property is satisfied. 
Recall that $\delta_\bT$, the transition function of $\bT$, is of the form 
 $\delta_\bT\colon \cS\times X \to \textrm{Finite Subsets}(\cS\times T^*)$.  We use $\pi_T$ to denote the evaluation from $T^*=\{t,t^{-1}\}^*$ to $\bZ$.
 Recall our convention of Remark \ref{rem: labelled edges} as we will define the new transitions of $\bM$ as  edges labelled by words.
For each $\dagger \in \{0, 1\}$, $s\in \cS$ and $x\in X\cup \{\epsilon\}$, and each edge in $\bT$ from $s$ to $r$ with label $x/u$ we now define a corresponding edge in $\bM$ starting from $(s,\dagger)$ and ending in $(r, \dagger')$ with label $xv$.
The number $\dagger'\in \{0,1\}$ and the word $v\in Z^*$ are defined below depending on the values of $\pi_T(u)$ and $\dagger$.

Suppose that   $\pi_T(u)=2k+\eta$ with $\eta\in \{0,1\}$. We would like to have the property that $\dagger'=\pi_T(u)+\dagger+ 2\pi_Z(v).$ To do so, we define $\dagger'$ and $v$ as follows.

\begin{itemize}
    \item if $\dagger =0$ then $\dagger' = \eta$ and $v=z^{-k}$,
    \item if $\dagger =1$ and $\eta =0$ then $\dagger'= 1$ and $v=z^{-k}$,
     \item if $\dagger =1$ and $\eta =1$ then $\dagger'= 0$ and $v=z^{-k-1}$.
\end{itemize}

Note that when we write $z^{-k}$, we mean the word in $Z$ which is either a $k$-repetition of $z$ or $z\inv$. 

There are no more states or transitions (i.e the transitions not described go to the emptyset) and this completes the description of $\bM$.

We now check that the Balancing property holds.
Let  $w\in (X\sqcup Z)^*$ and $(\alpha,\dagger)\in \delta_M((s_0,0),w)$.
Then $w$ is the label of a path in $\bM$ from $(s_0,0)$ to $(\alpha,\dagger)$.
As each label of each edge of $\bM$ is of the form $xv$ with $x\in X\cup\{\epsilon\}$ and $v\in \{z,z^{-1}\}^*$, we can write 
 $w$ as  $x_1v_1x_2v_2\dots x_\ell v_\ell\in \cL$ such that $x_i\in X\cup \{\epsilon\}$, $v_i\in Z^*$, and the sequence $(x_iv_i)_i$ is the sequence of the labels of the edges in the path in $\bM$ defined by $w$.
We check the Balancing property by induction on $\ell$.

For $\ell=1$, we have that $x_1v_1$ labels an edge of $\bM$ starting at $(s_0,0)$ and ends in some $(s,\dagger)$.
Such edge comes from an edge in $\bT$  from $s_0$ with label $x_1/u$ to $s$.
Moreover, we have that if $\pi_T(u)=2k+\eta$ with $\eta\in \{0,1\}$ then $v_1= z^{-k}$ and $\dagger = \eta$.
Observe that  $\tau'(\pi(w))=\tau'(\pi(x_1v_1))=\tau(\pi(x_1))-2k=2k+\dagger-2k=\dagger.$

We now show that the case $\ell$ implies the case $\ell+1$.
By induction hypothesis $(s,\dagger)\in \delta_\bM((s_0,0), x_1v_1\dots x_\ell v_\ell)$ and  $\tau'(\pi(x_1v_1\dots x_\ell v_\ell)) = \dagger$.
Then, $x_{\ell+1}v_{\ell+1}$ labels and edge of $\bM$ starting at $(s,\dagger)$ and ends in some $(s',\dagger')$.
Such edge comes from an edge in $\bT$ from $s$ with label $x_{\ell+1}/u$ to $s'$. 
If $\pi_T(u)=2k+\eta$ with $\eta\in \{0,1\}$ then $v_{\ell+1}$ and $\dagger'$ are determined by $\dagger$, $\eta$ and $k$.
We have 
\begin{align*}
\tau'(\pi(w)) & = \tau(\pi(x_1x_2\dots x_\ell x_{\ell +1})) + 2 \pi_Z(v_1\dots v_{\ell+1}) \\
& =\tau(\pi(x_1x_2\dots x_\ell)) + \pi_T(u)  + 2 \pi_Z(v_1\dots v_{\ell})+ 2\pi_Z( v_{\ell+1})\\
& = \dagger + \pi_T(u)  + 2 \pi_Z( v_{\ell+1}) = \dagger'
\end{align*}
Therefore the Balancing property holds.

Let $\cL$ be the language accepted by $\bM$. It was observed in the paragraph after defining the Balancing property that $\pi(\cL)\subseteq P$.

On the other hand, to see that $P \subseteq  \pi(\cL)$, observe first that given $(g,n)\in P$ there is some $w\in \bT^{-1}(T^*)$ such that $\pi_X(w)=g$.
Suppose that a possible way for $\bT$ to process $w=x_1\dots x_\ell$ is a sequence of states of $\bT$ with outputs in $T^*$ as follows $$s_0, \; (s_1,u_1)\in \delta_\bT(s_0,x_1),\;(s_2, u_2)\in \delta_\bT(s_1,x_2),  \dots,\; (s_\ell,u_\ell)\in \delta_\bT(s_{\ell-1},x_{\ell})$$ with $s_i\in \cS$.
Consider the word $${w}'=x_1z^{\nu_1}x_2z^{\nu_2}\dots x_\ell z^{\nu_\ell} \text{  and  the sequence  } \dagger_0=0,\dagger_1,\dots, \dagger_\ell \text{ in } \{0,1\}$$  where the $\nu_i$'s and $\dagger_i$'s are determined by $\dagger_{i+1}= \tau(\pi(x_1 \dots x_i))+ \dagger_i + 2(\nu_1+\dots +\nu_i)$.
Then, by the definition of the transitions of $\bM$ we see that,  $(s_{i+1},\dagger_{i+1})$ is the end point of an edge of $\bM$ from $(s_i,\dagger_i)$ with label $x_{i+1}z^{\nu_i+1}$.

Let now $m= n- \sum \nu_i$.
Observe first that $\pi(w'z^m)=(g,n)$.
Indeed, $\pi(w'z^m)=\pi(x_1\dots x_\ell)\pi(z^{\nu_1}\dots z^{\nu_\ell}z^m)=(g,m +\sum \nu_i)=(g,n)$.
Finally, observe that $w'z^m\in \cL$.
By the Balancing property and construction of $w'$ we have that $(s_\ell,\dagger_\ell)\in \delta_\bM((s_0,0),w') $ with $\tau'(\pi(w'))=\dagger_\ell\in \{0,1\}$. 
As $(g,n)\in P$ and $\tau'((g,n))= \tau'(\pi(w'))+2m=\dagger_\ell +2m>0$ we have that $m\geq 0$ if $\dagger_\ell =1$ and $m>1$ if $\dagger_\ell=0$.
We have that if $m=0$, $(s_\ell,\dagger_\ell)$ is an accepting state, and if $m>0$ then $\delta_\bM((s_\ell,\dagger_\ell), z^m)=f$.
In either case $\delta_\bM((s_\ell,\dagger_\ell), z^m)\subseteq \delta_\bM((s_0,0),w'z^m)$ contains an accepting state and $w'z^m\in \cL$.
\end{proof}

\begin{figure}[ht]
\begin{center}
\resizebox{\textwidth}{!}{
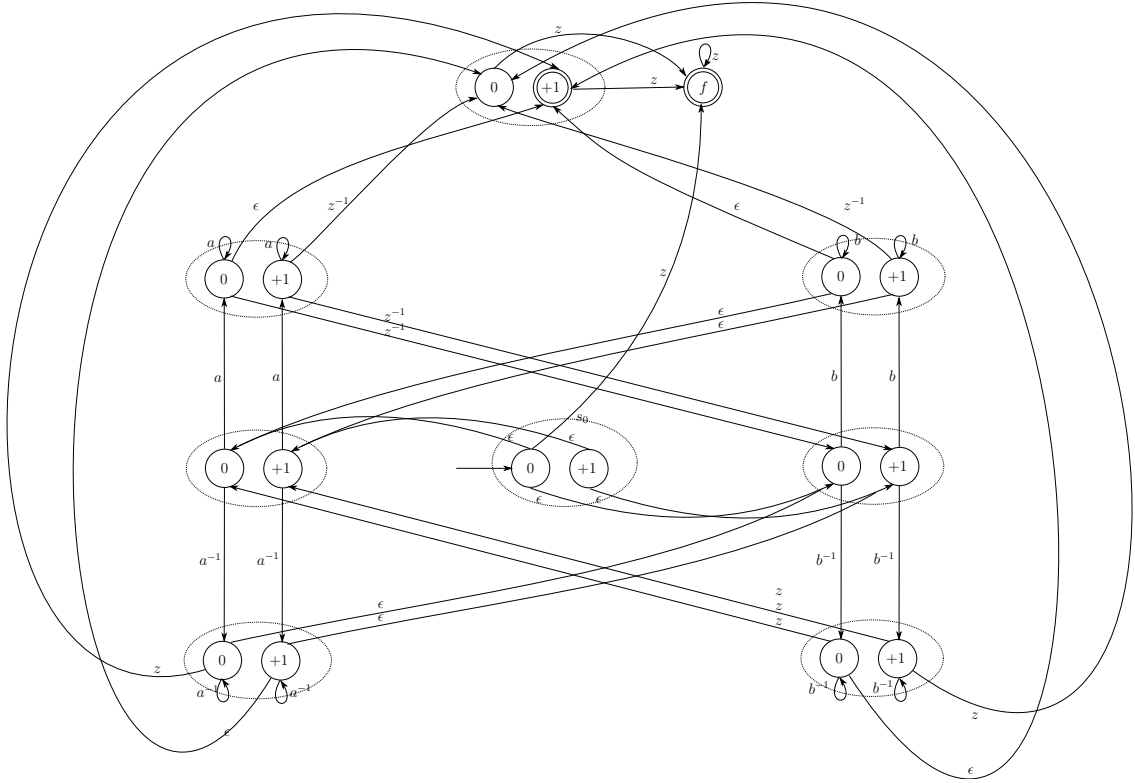
}
\end{center}
\caption{A finite state automaton for $F_2 \times \bZ$. It looks similar to Figure \ref{fig: transducer-f2}, with each state doubled and extra arrows added. The state $(s_0,1)$ could have been deleted as it has no incoming arrow, but we have chosen to keep it in the picture as it is an state and arrows that appear from the construction.}
\label{fig: reg-trans}
\end{figure}

\begin{thm}\label{thm: main}
Let $G_1,G_2, \dots G_n$ be finitely generated groups with a common subgroup $C$ such that each $G_i$ admits a $\Reg$-positive cone relative to $C$ and $C$ admits a $\Reg$-positive cone.
Let $G$ be the free product of the $G_i$'s amalgamated over $C$. Then $G\times \bZ$ admits a $\Reg$-left-order.
\end{thm}
\begin{proof}
By Proposition \ref{prop: free-prod-cf}, $G$ admits a $\tau$-transducer where $\tau$ is  an ordering quasi-morphism $\tau\colon G\to 2\bZ+1\cup \{0\}$ with kernel $C$.
Now by Theorem \ref{thm: key-embedding}, $G\times \bZ$ has $\Reg$-positive relative to $C\times \{0\}$.
Finally, since $C\cong C\times \{0\}$ has $\Reg$-positive cones, we get from Lemma \ref{lem: extend from convex} that $G\times \bZ$ has relative positive cones.
\end{proof}

\begin{cor}\label{cor: A*BxZ}
Let $A, B$ be groups admitting $\Reg$-left-orders. 
Then $(A*B)\times \bZ$ admits $\Reg$-left-orders.
\end{cor}

Also from the previous theorem, and by the discussion prior to  Corollary \ref{cor: BS(1,m;1,n)} we have
\begin{cor}
For $n,m\geq 1$, the group  $BS(1,m;1,n)\times \bZ$ has a $\Reg$-positive cone. 
\end{cor}

Another interesting application of the results of this paper, is the following.

\begin{thm}\label{thm: free-by-ZxZ}
Suppose that $G$ is a (non-abelian finitely generated free)-by-$\bZ$ group. 
Then, no lexicographic left-order $\prec$ on $G$ where $\bZ$ leads is regular.

However, there is a lexicographic left-order  on $G$ where $\bZ$ leads  that  is one-counter and extends to regular left-order on $G\times \bZ$.
\end{thm}
\begin{proof}
From Proposition \ref{prop: convex implies L-convex}, if $G$ admits a regular lexicographic left-order where $\bZ$ leads, then there will be a regular order on a finitely generated free group, contradicting  the theorem of Hermiller and \u{S}uni\'{c} \cite{HS} that says that non-abelian free groups do not admit regular positive cones.

Suppose that $f\colon G\to \bZ$ is a surjective homomorphism with kernel $F_n$, a free group of rank $n$. By Proposition \ref{prop: free-prod-cf} (see also \cite{Sunik13,Sunik13b}), there is  an ordering-quasimorphism  $\tau\colon F_n \to 2\bZ+1 \cup \{0\}$ with trivial kernel admitting a $\tau$-transducer. 
By Proposition \ref{prop: transducer imply one-counter}, $F_n$ admits a one-counter positive cone, and by  Proposition \ref{prop:extensions}, since $\bZ$ has regular orders, $G$ has a one-counter positive cone.

By Theorem \ref{thm: main},  we see that $F_n\times \bZ$ has a regular positive cone.
Note that $F_n\times \bZ$ is the kernel of $\tilde{f}\colon G\times \bZ\to \bZ$ given by $(g,n)\mapsto f(g)$.
As $\bZ$ has regular positive cones, Proposition \ref{prop:extensions} guaranties that $G\times \bZ$ has regular positive that is lexicographic with leading factor the quotient, when viewing $G\times \bZ$ as a ($F_n\times \bZ$)-by-$\bZ$ extension.
The restriction of this order on $G$ is still lexicographic.
\end{proof}

Some families of groups that are known to be (finitely-generated free)-by-cyclic are provided in the next corollary. We first recall the notion of Artin group.

\begin{defn}[Artin groups and defining graphs]
Let $\Gamma$ be a finite simplicial graph with edges labeled by integers greater or equal to two. 
We associate to $\Gamma$ a group $A(\Gamma)$ whose presentation has generators corresponding to the vertices of $\Gamma$ and the relations are of the form
$$\underbrace{aba\dots}_{n \text{ letters }} = \underbrace{bab\dots}_{n \text{ letters }}$$
 where $\{a,b\}$ is an edge of $\Gamma$ labeled with $n$. 
 The graph $\Gamma$ is called the \emph{defining graph} of the \emph{Artin group} $A(\Gamma)$.
\end{defn}

\begin{cor}
Let $A(\Gamma)$ be an Artin group whose defining graph $\Gamma$ is a tree. Then $G$ has one-counter left-orders and $G\times \bZ$ has regular left-orders.
\end{cor}

\begin{proof}
By a result of Hermiller and Meier \cite{HermillerMeier}, $A(\Gamma)$ admits a short exact sequence of the form $1 \to F_m \to A(\Gamma) \to \mathbb{Z} \to 1$ (and in fact  $m = \sum_{e_i \in \Gamma} (n_i - 1)$, where $n_i$ is the label of the edge $e_i$). Hence, $A(\Gamma)$ falls into the hypothesis of Theorem \ref{thm: free-by-ZxZ} and the conclusion follows.
\end{proof}

\begin{defn}
A right-angled Artin group is an Artin group whose defining graph only has edges with label 2.
\end{defn}

\begin{cor}\label{cor: RAAGS}
Let $G$ be a right-angled Artin group based on a connected graph with no induced subgraph isomorphic to $C_4$ (the cycle with $4$ edges) or $L_3$ (the line with 3 edges).
Then $G$ has  regular left-orders.
\end{cor}
\begin{proof}
The proof is by induction on the size of the defining graph $\Gamma$.
If the graph has one vertex, then $G\cong \bZ$ and we know that $\bZ$ only has regular left-orders.
Now suppose that the defining graph has more than one vertex.
Droms \cite[Lemma]{Droms} observed that in that case there is a vertex connected to all other vertex of the graph $\Gamma$. 
That is $G\cong \bZ\times H$ where $H$ is right-angled Artin group based on a graph $\Gamma'$ that contains no induced subgraph isomorphic to $C_4$ or $L_3$. 
Note that $\Gamma'$ has fewer vertices that $\Gamma$.
Then, by induction, each connected component of $\Gamma'$ defines a subgroup of $H$ that has regular left-orders.
If $\Gamma'$ is connected, then $H$ has regular left-orders and so does $G$.
If $\Gamma'$ is not connected, then $H=A(\Gamma')$ is a free product of groups having regular left-orders. By Proposition \ref{prop: free-prod-cf} and Theorem \ref{thm: key-embedding} we get that $H\times \bZ$ has regular left-orders.
\end{proof}

\begin{rem}
It follows from \cite[Theorem 4]{HS}, that there is no positive cone $P$ on a right-angled Artin group defined over a graph of diameter $\geq 3$ such the set of  all geodesic words in the standard generating set that represent elements of $P$ form a regular language.

We remark that all the defining graphs of considered in  Corollary \ref{cor: RAAGS} have diameter at most two.
We also point out that the words of the regular left-orders constructed for $F_2\times \bZ$ do not induce (uniformly quasi)-geodesic paths.
\end{rem}

\appendix
\section{Appendix: pre-images of positive cones}
\label {appendix}
In this appendix, we explore the following definition.
\begin{defn}
Let $\cC$ be a class of languages.
Let $G$ be a group finitely generated by $(X,\pi)$ and $\prec$ a left-order on $G$.
We say that $\prec$ is a {\it $\cC$-preimage left-order} if the language $\pi^{-1}(P_\prec)\in \cC$. 
\end{defn}
Note that the difference between a $\cC$-left-order and a $\cC$-preimage left-order is that given an element $g \in P_\prec$, a $\cC$-left-order does not contain all words which map to $g$, but has at least one. 
For a given class $\cC$, a left-order might be a $\cC$-left-order but not a $\cC$-preimage left-order.
In fact, for the lower classes in the Chomsky hierarchy such as $\Reg$ and $\CF$, we will see that there are few examples of $\cC$-preimage left-orders.  
On the other hand, for the class of recursively enumerable languages, the highest in the Chomsky hierarchy, we will see that admitting a $\cC$-left-order and is equivalent to admitting $\cC$-preimage left-orders if $G$ is finitely presented.

The first observation follows from Lemma \ref{lem: independence gen set}.
\begin{lem}\label{lem: C-preim indep gen set}
If $\cC$ is class of languages closed under inverse homomorphisms, then having a $\cC$-preimage left-order is independent of generating set.
\end{lem}

One of the advantages of studying $\cC$-preimage left-orders compared to $\cC$-left-orders is that being a $\cC$-preimage left-order is preserved when passing to subgroups.
\begin{lem}\label{lem: preimage inheritance}
Let $\cC$ a class of languages closed under inverse homomorphism and intersection with regular languages.

Let $H$ be a finitely generated subgroup of $G$.
If $\prec$ is a $\cC$-preimage left-order on $G$ then the induced order on $H$ is also  a $\cC$-preimage left-order.
\end{lem}
\begin{proof}
By the previous lemma, we can assume that we have a finite generating set $(X,\pi)$ of $G$, and that there is a subset of $Y$ of $X$ such that $(Y,\pi_Y=\pi|_{Y^*})$ is a generating set for $H$. 
Thus, if $P$ is a positive cone of $G$ with $\pi^{-1}(P)$ in the class $\cC$, we get that $\pi_Y^{-1}(H\cap P)= \pi^{-1}(P)\cap Y^*$ is in $\cC$.
\end{proof}

Now we will see that the classes of $\Reg$-preimage left-orders and $\CF$-preimage orders are quite limited.

Let $(X,\pi)$ be a finite generating set of a group $G$.
Recall that the set $\WP(G,X)= \{w\in X^* : \pi(w) = 1\}= \pi^{-1}(\{1\})$ is called the {\it the word problem} of $G$ (with respect to $X$).
A classic result of Anisimov \cite{Anisimov} states that $\WP(G,X)$ is regular if and only if $G$ is finite.

\begin{prop}\label{prop: regular-preimage}
A finitely generated group $G$ has a $\Reg$-preimage left-order if and only if it is trivial. 
\end{prop}
\begin{proof}
Let $(X,\pi)$ be finite generating set for $G$.
If $G$ is trivial, the positive cone is the empty set.

Suppose that $P$ is a $\Reg$-preimage positive cone. 
That is, $\pi^{-1}(P)$ is a regular language.
By Lemma \ref{lem: negative cone in cC}, we have that $\pi^{-1}(P^{-1})$ is also a regular language.
Since $WP(G,X)=X^*\setminus (\pi^{-1}(P^{-1})\cup \pi^{-1}(P))$, we get that  $WP(G,X)$ is regular, and by Anisimov's theorem $G$ must be finite.
Finally, a finite left-orderable groups must be trivial. 
\end{proof}

We now move on to the case  of $\CF$-preimage left-orders. 
We start showing that this class contains free groups.

\begin{prop}
Finitely generated free groups have a $\CF$-preimage left-orders.
\end{prop}
\begin{proof}
Let $F$ be a finitely generated free group with generating set $X$. From Corollary \ref{cor: free prod admit one-counter order} there exists a $\onecounter$-language $\cL\subseteq X^*$ representing the positive cone $P$ of $F$.
In fact, $\cL$ can be taken to be a sublanguage of reduced words over $X$ (see Proposition \ref{prop: free-prod-cf} or Figure \ref{fig: transducer-f2} for the rank 2 free group case).

Let $e$ be a symbol disjoint from $X$. Define $\phi : (X \cup \{e\})^* \to X^*$ to be the monoid homomorphism sending $e$ to the empty word. Then $\phi^{-1}(\cL)$ is the language of words in $w \in \cL$ with arbitrary insertions of symbols $e$ in between its letters. Let $s$ be a substitution (in the sense of Hopcroft, Motwani, and Ullman \cite{HopUll}) such that $s(e) = \text{WP}(F)$, meaning we replace the symbol $e$ with a word over $X$ which is equal to the identity in $F$, and $s(x) = x$ for $x \in X$ otherwise. 

We claim that $s(\phi^{-1}(\cL))$ is a context-free language which is the preimage of $P$ under $\pi_X$. The language is context-free by closure properties of context-free languages under substitution and inverse homomorphism (see for example \cite[Theorem 7.23 and Theorem 7.30]{HopUll}). Moreover, $\pi_X(s(\phi\inv(\cL))) = \pi_X(\cL)$ by construction, as each $e$-substitution is a word which is equal to the identity. Therefore, $s(\phi\inv(\cL)) \subseteq \pi_X\inv(P)$. 
For the converse, observe that for every positive element $g \in P$, there is a reduced word in $w \in \cL$ such that $\pi(w) = g$. 
The preimage of $g$ under the evaluation map $\pi_X$ is $w$ with arbitrarily many insertion of words which are equal to the identity. Therefore, $\pi\inv_X(g) \in s(\phi\inv(\cL))$, so $\pi\inv(P) \subseteq s(\phi\inv(\cL))$. 
\end{proof}

However, the class of $\CF$-pre-image left-orders is still very limited. 
Indeed, all finitely generated abelian subgroups of a group admitting a $\CF$-preimage positive cone must be cyclic.
\begin{prop}\label{prop: Z2 is not CF-preimage}
$\mathbb{Z}^2$ does not have  $\CF$-preimage left-orders.
In particular, if $G$ has a $\CF$-preimage left-order,  then it does not contain $\bZ^2$ as a subgroup.
\end{prop}

\begin{proof}
Positive cones of $\bZ^2$ are very well understood. See for example \cite[Section 1.2.1]{GOD} and references therein.
We view $\bZ^2$ as subspace of $\bR^2$.
A positive cone in $\bZ^2$ is determined by a line in $\bR^2$ through the origin, and then selecting one of the half-space (excluding the origin) delimited by the line. If the slope of the line is rational, causing some elements of $\bZ^2$ to lie on the line, then one has to choose half of the line (starting from but excluding the origin) to be in the positive cone as well.

\begin{figure}[ht]
\begin{center}
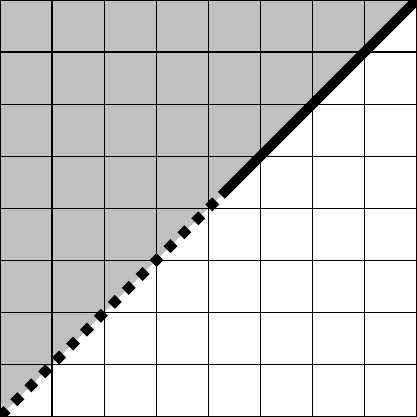
\end{center}
\caption{$\bZ^2$ with rational line and choices of orientation}
\label{fig: zsquare}
\end{figure}

By Lemma \ref{lem: C-preim indep gen set} we may choose any generating set of $\bZ^2$ for this proof. Take $a$ and $b$ to be standard basis vectors, in the $x$ and $y$ direction respectively.  
We use $A$ and $B$ to denote $a^{-1}$ and $b^{-1}$ respectively.

Suppose that the half-space defining the positive cone is $\texttt{y}> \lambda \texttt{x}$, where $\lambda \in \bR$. The case $\texttt{y}<\lambda \texttt{x}$ is analogous.
If $\lambda\in \mathbb{Q}$, we choose $\texttt{y}=\lambda \texttt{x}$, $\texttt{x}>0$ to be in the positive cone. 
Again, the case $\texttt{x}< 0$ is analogous. 

Therefore, we can assume that words evaluating to the positive cone consists on
$$\left\lbrace
w\in \{a,b,A,B\}^* \,\middle|  
\begin{array}{c} 
\text{ either }
(\sharp_b(w)-\sharp_B(w))>\lambda (\sharp_a (w)- \sharp_A (w))  \text{ or }
\\

((\sharp_b(w)-\sharp_B(w))=\lambda (\sharp_a (w)- \sharp_A (w)) \text{ and } (\sharp_a (w)- \sharp_A (w))>0 )
\end{array}
\right\rbrace.$$

Suppose the previous set is a $\CF$ language. Then if we take the intersection with the regular language $a^*b^*A^*B^*$ we get that $$\cL= \{ a^\alpha b^\beta A^\gamma B^\delta \,|\, (\beta-\delta)> \lambda(\alpha-\gamma)\text{ or } ((\beta-\delta)= \lambda(\alpha-\gamma)\text{ and } (\alpha-\gamma)>0)\}$$
is $\CF$. 
We now use the pumping lemma for $\CF$-languages (see \cite{HopUll} for example)  to derive a contradiction. 
Suppose that the pumping length is $p$ and take a word 
 $s=a^c b^{d}  A^e  B^f\in \cL$ with $c,d,e,f>p$.
The pumping lemma says that $s$ can be factorized  as $s=uvwxy$ (i.e. $u,v,w,x,y$ are all subwords of $s$),  with $\ell(vwx)\leq p$, and with at least $v$ or $x$ is non-empty, such that for all $i\geq 0$ $uv^iwx^i y$ belongs to $\cL$.
Let $\texttt{v}= \pi(v)+\pi(x)$ be a vector in $\bZ^2$.
Let $\texttt{p}= \pi(s)$ a point of $\bZ^2$.
Note that the points $\pi(uv^iwx^iy)$ for $i\geq 0$ lie on the line $L= \{\texttt{p}+ \mu \texttt{v} : \mu \in \mathbb{R}\}$ that goes through $\pi(s)$ 
and there are points a $\pi(uv^iwx^iy)$ on both sides of $L\setminus \{\pi(s)\}$, for instance $\pi(uwy)$ and $\pi(uv^2wx^2y)$. Remark that the line $L$ has rational slope.

Now  take $s=a^c b^{d}  A^e  B^f$ with $c,d,e,f>p$ and such that  $\pi(s)$ is very close to the line $\texttt{y}=\lambda \texttt{x}$
 (or on the line if $\lambda$ is rational) and with the property that for any  line $L$ going through $\pi(s)$ there is a half-line of $L\setminus \{\pi(s)\}$ such that all rational points on this half-line are outside of $P$.
If $\lambda$ is not rational, then we can find such point $\pi(s)$ by taking a rational approximation of $\lambda$.
If $\lambda$ is rational, we take $\pi(s)$ to be the first point with integer coordinates on the set $\{(\texttt{x},\texttt{y}) \mid \texttt{y}=\lambda \texttt{x}, \texttt{x}>0\}$.

Then for that $s$, some choice of $i$ on the pumping lemma will give an element not in $\cL$. Indeed, evaluating  pumped words with $i=0$ and $i=2$, we get two different points on $\bZ^2$ that lie on a line going through $\pi(s)$, and moreover these points lie in different components of the line minus $\pi(s)$, thus some of them are in the negative cone. This gives the desired contradiction.

From Lemma \ref{lem: preimage inheritance} we get that no group with $\CF$-preimage left-orders can contain $\bZ^2$.
\end{proof}

The celebrated result of Muller and Schupp \cite{MullerSchupp83} states that  the word problem of a group is context-free if and only if $G$ is virtually free.
 Unfortunately, the argument of Lemma \ref{prop: regular-preimage}, does not generalize directly to context-free languages, since this class is not closed under complement. Thus, it is natural to ask the following.
\begin{prob}
Is there a non-free group admitting a $\CF$-preimage left-order?
\end{prob}
A simpler, but related question is the following.
\begin{prob}
Is there a non-cyclic group admitting a $\onecounter$-preimage left-order?
\end{prob}

Recall that the class of languages that are complements of a context-free languages are called {co-context-free} and we denote it by {co-$\CF$}. The class of groups for which $WP(G,X)$ is co-$\CF$ was first studied by Holt, Rees, R\"{o}ver, and Thomas \cite{HRRT}. We have the following easy observation.

\begin{prop}
Let $G$ be finitely generated group with $\CF$-preimage left-orders.
Then the word problem of $G$ is in the class co-$\mathcal{CF}$.
\end{prop}
\begin{proof}
Since $\CF$ is closed under reversal and homomorphisms, we get that the full language of $\prec$-negative words is $\CF$. Since the class $\CF$ is closed under union, $WP(G)$ is co-$\mathcal{CF}$.
\end{proof}
The class of groups of {co-$\CF$} word problem is closed by taking finite direct products, taking restricted standard wreath products $\oplus H\rtimes Q$  with $Q$ having  context-free word problem, passing to finitely generated subgroups, and passing to finite index overgroups \cite{HRRT}. In \cite[Theorem 13]{HRRT} it is also shown that a Baumslag-Solitar group has a co-$\CF$ word problem  if and only if it is virtually abelian. Thus, we obtain the following.

\begin{cor}\label{cor: BS is not CF-pre}
Baumslag-Solitar groups do not admit $\CF$-preimage left-orders.
\end{cor}
Note that $\bZ^2$ embeds into every Baumslag-Solitar group, except $BS(1, n)$ with $n\neq \pm 1$ \cite[Proposition 7.11]{Levitt}, thus,  Corollary \ref{cor: BS is not CF-pre} does not follow from Proposition \ref{prop: Z2 is not CF-preimage} on the solvable case.

Lehnert and Schweitzer \cite{LS} showed that Thompson's group $V$ has co-$\CF$ word problem, and a conjecture of Lehnert can be formulated as all groups with co-$\CF$ word problem are finitely generated subgroups of $V$  \cite{BMN}.
Thus, a potential place to look for examples of groups with $\CF$-preimage left-orders are left-orderable subgroups of $V$ that do not contain $\bZ^2$.

\paragraph{Positive cones whose preimage is recursively enumerable.}
We now recall some definitions and results that we will need.

\begin{defn}
A {\it recursively enumerable language} is a formal language for which there exists a Turing machine  which enumerates all accepted strings of the language.
We denote the family of recursively enumerable languages by $\RecEnum$.

\end{defn}
We will not give a formal definition of a Turing machine. It can be found in  \cite{HopUll}, for example.
We will use that if $\cL$ is $\RecEnum$ is recursively enumerable, then there is an algorithm that lists all elements of $\cL$. On the other hand, if there is an algorithm that decides if a given word $w$ is in $\cL$, then $\cL$ is in $\RecEnum$. Note that for such algorithm, on an input $w\in \cL$  always halts and decides correctly, where for $w\not\in \cL$ it never says that $w$ lies in $\cL$ but it might not halt (thus we can not deduce that $w$ does not belong to $\cL$ in that case).

The following is well-known. See \cite{HopUll}.

\begin{prop}
The class $\RecEnum$ is a full AFL and it is closed under reversal.
\end{prop}

Recall that a group is recursively presented if it admits a presentation $\langle X \mid R \rangle$ where $X$ is finite and $R$ is a recursively enumerable language over $X^*$.

We now observe  the following.
\begin{lem}\label{lem:RecEnum preim LO= RecEnum LO}
Let $G$ be a recursively presentable group. Then $G$ admits $\RecEnum$-left-orders if and only if it admits $\RecEnum$-preimage left-orders.
\end{lem}
\begin{proof}
Let $(X,\pi_X)$ be a finite generating set of a group $G$.
If $\prec$ is a  $\RecEnum$-preimage left-order, then $\prec$ is a $\RecEnum$ left-order.

Assume that $\prec$ is a $\RecEnum$ left-order and let $P$ be the corresponding positive cone.
Let $\cL\subseteq X^*$ be in $\RecEnum$ so that $\pi(\cL)=P$.
Since $G$ is recursively presented, the word problem $WP(G,X):= \pi^{-1}(\{1\})$, is in $\RecEnum$.

Let $w\in X^*$ such that $\pi(w)\in P$. 
Then there is $u\in \cL$ such that $\pi(w)=\pi(u)$ and therefore, $wu^{-1} \in WP(G,X)$.
This implies that $w =_{F(X)} vu$ with $v\in WP(G,X)$ and $u\in \cL$, where the equality is as group elements of the free group, not as words. 

Since $WP(G,X)\cdot \cL$ is recursively enumerable, and freely reducing words can be computed by a Turing machine, we can decide if $w\in \pi^{-1}(P)$ by checking if the freely reduced version of $w$ is equal to a freely reduced version of some word in $WP(G,X)\cdot \cL$ whose elements can be listed by an algorithm.
\end{proof}

We note the following special case.
\begin{cor}\label{cor: word problem}
If $G$ is recursively presented and has a $\RecEnum$-left-order, then $G$ has solvable word problem.
\end{cor}
\begin{proof}
Recall that having solvable word problem is independent of generating sets. 
Fix $(X,\pi)$ a generating set for $G$.
We need to show that if $\cL\subseteq X$ is a $\RecEnum$ language representing a positive cone, and $G$ is recursively presented, then  there is an algorithm that decides if given a word in the generators, this word represents the trivial element or not. 

For a recursively presented group,  $WP(G,X)=\pi^{-1}(\{1\})$ is $\RecEnum$. 

Since the class of recursively enumerable languages are closed under reversal and homomorphism and union (see Proposition \ref{prop: AFL}), using  Lemma \ref{lem: negative cone in cC}, it is easy to see that there is a recursively enumerable language over $X$,  such that $\pi^{-1}( P \sqcup P\inv)$ is $\RecEnum$.

Thus, given a word, we can check if it belongs to $\pi^{-1}(\{1\})$ or to $\pi^{-1}(G\setminus \{1\})$ since both are recursively enumerable languages, and hence the word problem is solvable.
\end{proof}

We note that the converse of Corollary \ref{cor: word problem} is false. 
First, not all left-orderable groups admit $\RecEnum$-left-orders. 
This has been recently proved by Harrison-Trainor \cite{HT} for left-ordererable groups and, in the bi-orderable case by Darbinyan \cite{Dar}. 
Moreover, the lack of $\RecEnum$-left-orders is not related to the solvability of the word problem since \cite[Corollary 2]{Dar} says that there exists a finitely presentable left-orderable group with solvable word
problem and without $\RecEnum$-left-orders.

\vspace{1cm}

\noindent{\textbf{{Acknowledgments}}} 

The three authors acknowledge Z. \u{S}uni\'{c} ans S. Hermiller for many fruitful discussions on this subject.  The three authors are also grateful to the referee/s for his/her comments and pointing out a simplication on the proof of Theorem \ref{thm: key-embedding}.

 Yago Antol\'{i}n and Hang Lu Su acknowledge partial support from the Spanish Government through grants number  MTM2017-82690-P, and through the ''Severo Ochoa Programme for Centres of Excellence in R\&{}D'' (SEV-2015-0554) and (CEX2019-000904-S)
 
 Hang Lu Su  has received funding from ``la Caixa'' Foundation (ID 100010434) with fellowship code LCF/BQ/IN17/11620066,
 from the European Union's Horizon 2020 research and innovation programme under the Marie Sklodowska-Curie grant agreement No. 713673.

Crist\'obal Rivas acknowledge the support from Fondecyt 1181548 and Fondecyt 1210155.

\noindent\textit{\\ Yago Antol\'{i}n,\\
Fac. Matem\'{a}ticas, Universidad Complutense de Madrid and \\ 
Instituto de Ciencias Matem\'aticas, CSIC-UAM-UC3M-UCM\\
Madrid, Spain\\}
{email: yago.anpi@gmail.com}

\noindent\textit{\\ Hang Lu Su\\
Instituto de Ciencias Matem\'aticas, CSIC-UAM-UC3M-UCM and\\
Dpto. de Matem\'aticas
Universidad Aut\'onoma de Madrid \\
Madrid, Spain\\}
{email: homeowmorphism@gmail.com}

\noindent\textit{\\ Crist\'obal Rivas\\
Dpto.\ de Matem\'aticas y C.C.\\
Universidad de Santiago de Chile\\
Alameda 3363, Estaci\'on Central\\ Santiago, Chile\\}
{email: cristobal.rivas@usach.cl}

\end{document}